\documentclass[a4paper, reqno, 11pt]{amsart}
\makeatletter

\@addtoreset{equation}{section}
\makeatother
\usepackage{setspace}
\usepackage{xcolor}
\usepackage{braket}
\usepackage{amsmath}
\usepackage{amssymb}
\usepackage{latexsym}
\usepackage{amsthm}
\usepackage{textcomp}
\usepackage{hyperref}
\usepackage{here}
\newtheorem{Thm}{Theorem}[section]

\newtheorem{Lem}[Thm]{Lemma}
\newtheorem{Prop}[Thm]{Proposition}

\newtheorem{Cor}[Thm]{Corollary}

\theoremstyle{definition}

\newtheorem{Ass}[Thm]{Assumption}
\newtheorem{Def}[Thm]{Definition}
\newtheorem{Rem}[Thm]{Remark}

\begin{document}

\title[Properties of power mean]{Properties of complex-valued power means of random variables and their applications}
\author[Y.Akaoka]{Yuichi Akaoka}
\thanks{A part of this paper consists of Y.~A.'s master's thesis \cite{Akaoka2020a}.}
\address{Department of Mathematics, Faculty of Science, Shinshu University}
\curraddr{Gunma bank}
\email{18ss101b@gmail.com}

\author[K. Okamura]{Kazuki Okamura}
\address{Department of Mathematics, Faculty of Science, Shizuoka University}
\email{okamura.kazuki@shizuoka.ac.jp}

\author[Y. Otobe]{Yoshiki Otobe}
\address{Department of Mathematics, Faculty of Science, Shinshu University}
\email{otobe@math.shinshu-u.ac.jp}

\subjclass[2000]{60F25, 60F15, 60F05, 26E60, 62F10,  62F12, 62E20}
\keywords{quasi-arithmetic mean; power mean; integrability; limit theorem; point estimation; Cauchy distribution}
\date{\today}
\dedicatory{}

\begin{abstract}
We consider power means of independent and identically distributed (i.i.d.) non-integrable random variables.   
The power mean is an example of a homogeneous quasi-arithmetic mean. 
Under certain conditions, several limit theorems hold for the power mean, similar to the case of the arithmetic mean of i.i.d.\ integrable random variables. 
Our feature is that the generators of the power means are allowed to be complex-valued, which enables us to consider the power mean of random variables supported on the whole set of real numbers. 
We establish integrabilities of the power mean of i.i.d.\ non-integrable random variables and a limit theorem for the variances of the power mean. 
We also consider the behavior of the power mean as the parameter of the power varies. 
The complex-valued power means are unbiased, strongly-consistent, robust estimators for the joint of the location and scale parameters of the Cauchy distribution.  
\end{abstract}

\maketitle

\section{Introduction}

It is important to consider heavy-tailed distributions, which appear in both theory and practice. 
However, they are not easy to handle since the law of large numbers fails for i.i.d.\ random variables with a non-integrable distribution such as the Cauchy distribution. 
For non-negative random variables, 
we can expect the law of large numbers for geometric and harmonic means of i.i.d.\ random variables by the arithmetic mean-geometric mean inequality and the geometric mean-harmonic mean inequality. 
Pakes \cite{Pakes1999} considered the the asymptotic behavior of the variances of the geometric and harmonic means of i.i.d.\ non-negative random variables. 

The geometric and harmonic means belong to a class of {\it quasi-arithmetic means} considered by Kolmogorov \cite{Kolmogorov1930} in his research on axioms of generalized means. 
A quasi-arithmetic mean has the form of $M^f_n (x_1, \dots, x_n) = f^{-1}(\frac{1}{n} \sum_{i=1}^{n} f(x_i))$ where $f$ is a function  called a generator. 
The arithmetic, geometric, and harmonic means are obtained by considering the cases that $f(x) = x$, $f(x) = \log x$, and $f(x) = 1/x$ respectively. 
Using the delta method in mathematical statistics, 
de Carvalho \cite{Carvalho2016} showed the central limit theorem for quasi-arithmetic means. 
Barczy and Burai \cite{Barczy2022} extended it to a more general framework called Bajraktarevi\'c means. 
The quasi-arithmetic mean is related to Fr\'echet means and fractional moments. 
However, there are few references which deal directly with quasi-arithmetic means of random variables. 

It is natural to consider heavy-tailed distributions {\it supported on $\mathbb{R}$} such as the Cauchy distribution.  
In some cases, especially when the sample contains outliers, modeling in a heavy-tailed distribution is more appropriate than modeling in the normal distribution. 
The results in \cite{Pakes1999, Carvalho2016, Barczy2022} are limited to non-negative random variables and are not applicable to this case. 
To solve this problem, 
the authors in \cite{Akaoka2021-1} introduced the notion of {\it complex-valued} quasi-arithmetic means of real numbers by allowing $f$ to take complex values. 
Then the authors established the asymptotic behavior of the variances of the geometric and harmonic means of random variables supported on $\mathbb{R}$, 
which are extensions of the results of \cite{Pakes1999}.  
Since we cannot adapt the arithmetic mean-geometric mean inequality or the geometric mean-harmonic mean inequality due to the extension of the domain to $\mathbb{R}$, 
it is difficult to investigate the integrability of $M^f_n$, which depends on the integrability of the random variable, the size of the sample $n$, and the generator $f$. 

This paper substantially develops the research in  \cite{Akaoka2021-1}.  
Our first contribution is to establish a more general result for the asymptotic behavior of the variances of $M^f_n$ of random variables. 
It is applicable to a large class of quasi-arithmetic means and contains the two main results of \cite[Theorems 2.1 and 3.1]{Akaoka2021-1} as corollaries. 
For the proof, we give a unified approach which differs from the ad hoc approaches taken in \cite{Akaoka2021-1}. 
Our results cannot be derived from the delta method and we need more delicate arguments than the derivation of the central limit theorem of $M^f_n$. 

Our second contribution is to study in detail the properties of $M^f_n$ of random variables in the case of {\it power means}, that is, $f(x) = x^p$ for $0 < |p| < 1$.   
In particular, we consider (i) the asymptotic behavior of the variances of $M^f_n$, (ii) the integrability of $M^f_n$, and, (iii) the behavior of $M^f_n$ as the power $p$ varies. 
We deal with both positive and negative powers. 
Negative power means interpolate between the harmonic and geometric means, and, positive power means interpolate between the geometric and arithmetic means. 
Our results are largely different depending on the sign of the power $p$. 
Perhaps contrary to our initial intuition, the case of negative powers is easier to deal with than that of positive powers. 
In the case of positive powers, $M^f_n$ is not integrable for some heavy-tailed distributions including the Cauchy distribution, and we consider some truncations of non-integrable terms.

All of our main results for quasi-arithmetic means are applicable to the Cauchy distribution, which is a canonical example of the heavy-tailed distribution. 
By our results, the negative power mean, as well as the geometric and harmonic means, work well as estimators of the joint of the location and scale parameters of the Cauchy distribution.  
The quasi-arithmetic mean $M^f_n$ has good integrability even when the sample size $n$ is small. 
The quasi-arithmetic means are $\sqrt{n}$-consistent, unbiased, strongly-consistent, robust estimators of the joint of the parameters, under the {\it complex parameterization} of Letac \cite{Letac1978} and  McCullagh \cite{McCullagh1996}. 

Our approach to the power mean leads to a novel estimator of the parameters of the {\it mixture} of the two Cauchy distributions with five unknown parameters. 
We give a strongly consistent and $\sqrt{n}$-consistent estimator in a closed form. 
Recently, Kalantan and Einbeck \cite{Kalantan2019} have considered this problem by using the EM algorithm, however, they focused on simulations and there are no mathematical guarantees. 

The rest of this paper is organized as follows. 
In Section \ref{sec:bkgd}, we give a more detailed background to this work.   
We discuss relationships with Fr\'echet means and fractional moments. 
In Section \ref{sec:asym-var}, we first give the notation used in this paper, and then state the first main result, which gives asymptotics for variances of $M^f_n$, in Theorem \ref{thm:gen-var-lim} and its proof. 
The following two sections are devoted to giving the second contribution. 
The cases of negative and positive powers are dealt with in Sections \ref{sec:neg} and \ref{sec:posi}, respectively. 
In Section  \ref{sec:pdsum}, we establish limit theorems for sums of products of random variables, which are not quasi-arithmetic means, but are naturally derived from the considerations of positive power means in Section 5. 
The last two sections are devoted to applications. 
In Section  \ref{sec:Cauchy}, we apply our results from Sections \ref{sec:asym-var} to \ref{sec:pdsum} to the Cauchy distribution. 
We thoroughly compare quasi-arithmetic means with other estimators of the parameters such as estimators depending on the order statistics and the maximum likelihood estimator. 
Finally, in Section \ref{sec:mix}, we deal with the mixture Cauchy model.

\section{Background}\label{sec:bkgd}

In this section, we give more detailed explanations for our motivations and related results. 

\subsection{Motivations}

Kolmogorov \cite{Kolmogorov1930} proposed  axioms of means 
and showed that if the  axioms of means hold for an $n$-ary operation on a set, 
then it has the form of a quasi-arithmetic mean.  
See \cite[Theorem 1.2]{Barczy2022} for the precise statement.   
In the study of quasi-arithmetic means, the generator $f$ is a real-valued function defined over an interval on $\mathbb R$ and it is strictly monotone and continuous on the interval. 
By de Finetti-Jessen-Nagumo's result, 
if the generator $f$ of a quasi-arithmetic mean is continuous, and homogeneous, specifically, 
\begin{equation}\label{eq:scale-equivariant} 
f^{-1} \left( \frac{1}{n} \sum_{j=1}^{n} f(a x_j)  \right) = a f^{-1} \left( \frac{1}{n} \sum_{j=1}^{n} f(x_j) \right), \ \ a, x_1, \cdots, x_n > 0,
\end{equation}
then, $f(x) = x^p, p \ne 0$ or $f(x) = \log x$. 
See Hardy-Littlewood-P\'olya \cite[p68]{Hardy1952}. 
Thus, not only the geometric and harmonic means but also power means are canonical examples of the quasi-arithmetic mean.

In \cite{Akaoka2021-1}, the authors dealt with the case that the generator of quasi-arithmetic means is given by $f(x) = \log(x+\alpha)$ or $f(x) = 1/(x+\alpha)$ for some $\alpha \in \overline{\mathbb H}$, 
where $\mathbb H$ is the upper-half plane and $\overline{\mathbb H}$ is its closure. 
We have changed the definition of the generator slightly by adding the complex number $\alpha$. 
Sections \ref{sec:neg} and \ref{sec:posi} of this paper deal with the case that 
\begin{equation}\label{eq:def-generator}
f(x) = f^{(\alpha)}_p (x) :=  (x+\alpha)^p, \ \ p \in [-1,1] \setminus \{0\}, \ \alpha \in \overline{\mathbb H}, 
\end{equation}
which we call here a {\it (complex-valued) power mean}.  
If $p=1$, which is the case of the arithmetic mean, then, regardless of the choice of $\alpha$, the quasi-arithmetic mean with generator $f$ is the arithmetic mean. 
The case of $p=0$ corresponds to the case of the geometric mean, that is, $f(x) = \log(x+\alpha)$. 
See Section \ref{sec:asym-var} for the definition of the power of complex numbers. 

There are two technical reasons for considering not only the case that $\alpha = 0$ or $\alpha \in \mathbb R$ but also the case that $\alpha \in \mathbb H$. 
One reason is that the case that $\alpha \in \mathbb H$ is easier to handle than the case that $\alpha = 0$ or $\alpha \in \mathbb R$ in terms of integrability. 
Indeed, if $p=-1$ and $\alpha \in \mathbb R$, which includes the case of the harmonic mean, then, $M^f_n$ follows the Cauchy distribution for each $n$ if $X_1$ follows the Cauchy distribution. 
See Remark \ref{rmk:harmonic} below for more details. 
Another reason is that the assumption that $\alpha = 0$ or  $\alpha \in \mathbb R$ is not appropriate for applications of the quasi-arithmetic mean to the estimation of the scale parameter of the Cauchy distribution. 
See  \cite[Corollary 2.6]{Akaoka2021-1} for more details. 
Many of our results depend on whether $\alpha \in \mathbb R$ or not. 
It is also natural to consider the case that $|p| > 1$ or $p \in \mathbb{C}$, but we do not deal with this case here.

We explain our results through the Cauchy distribution, which motivates our work, in an informal way. 
See Section \ref{sec:Cauchy} for more details. 
The Cauchy distribution is often used to formulate statistical models with heavy tails. 
Because of its heavy tails, we cannot define its expectation and variance, 
and it has no moment-generating functions. 
The arithmetic mean of a sample of any size from the Cauchy distribution has the same distribution as a sample of size one, so it is not applicable to estimating the location parameter.   
However, if we replace the arithmetic mean with a more general mean, then, we may be able to obtain the almost surely convergence. 

Now we make this intuition more precise. 
We assume that $X$ follows the Cauchy distribution with location $\mu \in \mathbb{R}$ and scale $\sigma > 0$. 
Let $\gamma := \mu + \sigma i$. 
Using the residue theorem, we can show that 
$E\left[X^p \right] = \gamma^p$ for $0 < |p| < 1$. 
Since $P(X < 0) > 0$, we allow the power $X^p$ to take {\it complex} values. 
In fact, $X^p \notin \mathbb{R}$ if $X < 0$. 
Let $X_1, X_2, \cdots$ be i.i.d.\ random variables following the Cauchy distribution with a complex parameter $\gamma \in \mathbb{H}$. 
Then, $\frac{1}{n} \sum_{j=1}^{n} X_j^p$ is an unbiased strongly-consistent estimator of $\gamma^p$. 
(For simplicity we assume that $\alpha = 0$.)
Since $E\left[|X_1^p|\right] < +\infty$,  
$\lim_{n \to \infty} \frac{1}{n} \sum_{j=1}^{n} X_j^p = \gamma^p$, almost surely.  
This is a version of the method of moments. 
By the strong law of large numbers and taking the power $1/p$, 
we obtain the strong law of large numbers for the power mean $\left(\frac{1}{n} \sum_{j=1}^{n} X_j^p\right)^{1/p}$, that is,  
$\lim_{n \to \infty} \left(\frac{1}{n} \sum_{j=1}^{n} X_j^p\right)^{1/p} = \gamma$, almost surely. 
We also see that the central limit theorem holds by the delta method. 
It is now interesting to consider more detailed properties of the power mean. 
We will show that for every $n \ge 2$, 
$\left(\frac{1}{n} \sum_{j=1}^{n} X_j^p\right)^{1/p}$ is integrable if $p \in (-1,0)$. 
By using a result in Section \ref{sec:neg}, 
we will show that $E\left[\left(\frac{1}{n} \sum_{j=1}^{n} X_j^p\right)^{1/p}\right] = \gamma$ for  $p \in (-1,0]$ and this means that the negative power mean is an unbiased estimator of $\gamma$.  
However, on the other hand, we can also show that for every $n \ge 2$, $\left(\frac{1}{n} \sum_{j=1}^{n} X_j^p\right)^{1/p}$ is not integrable if $p \in \{-1\} \cup (0,1]$. 
Our idea is to recover integrability by truncating a non-integrable term $\sum_{j=1}^{n} X_j$. 
We will show the truncated random variable $\left(\sum_{j=1}^{n} X_j^p\right)^{1/p} - \sum_{j=1}^{n} X_j$ is integrable 
and $E\left[\frac{1}{n^{1/p} -n}  \left(\left(\sum_{j=1}^{n} X_j^p\right)^{1/p} - \sum_{j=1}^{n} X_j\right) \right] = \gamma$  by using a result in Section \ref{sec:posi}.

\subsection{Related works}

In addition to \cite{Kolmogorov1930}, axiomatic treatments and properties of means have been considered for many years (e.g. \cite{Nagumo1930, Finetti1931, Aczel1948}). 
There are various definitions of means (see \cite{Bullen2003}) and accordingly 
we can consider generalized means of i.i.d.\ random variables. 
Generalized means are also related to parametric or nonparametric estimation in mathematical statistics. 
There are several techniques which are useful for generalized means of random variables. 
One technique is the {\it delta method} used in \cite{Carvalho2016, Barczy2022}. 
Berger and Casella \cite{Berger1992} found that the quasi-arithmetic mean can be regarded as a least squares estimate and the geometric and harmonic means appear in maximal likelihood estimates of the lognormal and inverse gamma distributions respectively. 
They also constructed confidence intervals by using the quasi-arithmetic mean as a point estimator.  
In the following subsections, we review Fr\'echet means and fractional moments, which are generalized means originated from statistics, and compare them with our complex-valued quasi-arithmetic means.

\subsubsection{Fr\'echet means}

The framework of Fr\'echet means is more general than the Bajraktarevi\'c mean in \cite{Barczy2022}.   
The Fr\'echet mean or the intrinsic mean of a Borel probability measure $P$ on a metric space $(M, d)$ is defined by a minimizer of the map $p \mapsto \int_M d(p,x)^2 P(dx)$ if it is uniquely determined.  
The Fr\'echet mean of the empirical measure for i.i.d.\ random variables has been considered by several authors. 
The motivations come from directional statistics and statistical shape theory. 
The uniqueness and consistency have been considered by Ziezold \cite{Ziezold1977, Ziezold1989, Ziezold1994}, Karcher \cite{Karcher1977}, Kendall \cite{Kendall1990},  Le \cite{Le1998, Le2001}, Kendall and Le \cite{Kendall2011}, and recently by Sch\"otz \cite{Schotz2021}. 
Bhattacharya and Patrangenaru \cite{Bhattacharya2002, Bhattacharya2003, Bhattacharya2005} established the central limit theorems for the intrinsic and extrinsic means. 
The techniques of differential geometry were used in these references.  
The results are summarized in the books of Bhattacharya and Bhattacharya \cite{Bhattacharya2012} and Kendall, Barden, Carne, and Le \cite{Kendall1999}. 

Recently, Kobayashi and Wynn \cite{Kobayashi2020} discussed the uniqueness of the intrinsic mean on empirical graphs, whose vertices consist of data points, in terms of metric geometry. 
As mentioned in  \cite[Section 1]{Berger1992}, the quasi-arithmetic mean can be regarded as the Fr\'echet mean of the empirical measure $P = \frac{1}{n} \sum_{i=1}^{n} \delta_{x_i}$, and the metric space $(M,d)$ where $M$ is the domain of $f$ and $d(x,y) := |f(x) - f(y)|$. 
As was considered in Itoh and Satoh \cite{Itoh2021}, it is also interesting to consider replacing the squared distance $d(p,x)^2$ appearing in the integral $\int_M d(p,x)^2 P(dx)$ with the Busemann function. 

The quasi-arithmetic mean plays an intermediate role between the arithmetic mean and the Fr\'echet mean.
Our framework, detailed in the following subsection, is contained in the framework of Fr\'echet means.
However, as we will see later, the quasi-arithmetic mean is easier to handle than the Fr\'echet mean, and its properties are quite similar to the arithmetic mean.

\subsubsection{Fractional moments}

When we deal with the method of moments, we often consider $k$-th moments for positive integer $k$.  
However, the framework of fractional moments, as well as our setting, are in a general category of the method of moments in \cite[Section 4.1]{Vaart1998}. 

We review some studies. 
Tallis and Light \cite{Tallis1968} used positive fractional moments to estimate the parameters of the {\it mixture} model of two one-parameter exponential distributions, 
improving the efficiency of moment estimators  compared to the consideration of integer moments by Rider \cite{Rider1961}. 
Later, From and Saxena \cite{From1989} used positive fractional moments for the mixture model of two scale families of positive samples. 
In contrast to \cite{Tallis1968}, \cite{From1989} obtained closed-form estimators.  

Mukherjee and Sasmal \cite{Mukherjee1984}  used positive fractional moments to estimate the shape and scale parameters of the Weibull distribution, 
and considered the optimal choice of two fractional moments such that the ratio of the asymptotic generalized variance of the maximum likelihood estimators to that of the moment estimators is maximized, by using the techniques of \cite{Tallis1968}.  
Mathai \cite{Mathai1991} considered positive and negative fractional moments of quadratic expressions of multidimensional normal random variables. 
Kozubowski \cite{Kozubowski2001} gave a method for estimating the parameters of the Linnik and Mittag-Leffler distributions based on fractional moments. 
These distributions are special cases of geometrically stable distributions and are applicable to modeling financial assets with heavy tails. 
It has recently been shown that the Mittag-Leffler distribution appears as a distributional limit of some quantities of elephant random walks (\cite{Bercu2022}). 

It is well-known that positive integer moments characterize any exponentially integrable distribution. 
Fractional moments are also used to characterize distributions of positive random variables. 
Lin \cite{Lin1992} showed that a distribution on positive reals is determined by every sequence of moments with positive fractional exponents satisfying certain conditions by using complex analysis.  
Some researches in this direction are \cite{Tagliani2003, Novi2003-1, Novi2003-2, Novi2005, Gzyl2006, Taufer2009, Gzyl2010-1, Gzyl2010-2}. 

Matsui and Pawles \cite{Matsui2016} considered fractional {\it absolute} moments with exponents between $1$ and $2$ for random variables with infinite variance and finite mean, by using relationships between the fractional absolute moments and Marchaud fractional derivatives. 
Their results are applicable to infinitely divisible distributions and compound Poisson processes. 
In probability theory, there are several types of research giving relationships between the fractional moments and the corresponding characteristic function or Laplace transform, which are reviewed in \cite[Section 1]{Matsui2016}. 
Recently, Mukhopadhyay et.\ al.\ \cite{Mukhopadhyay2021} derived the central limit theorem for the fractional positive moments of random variables when the sample distribution follows a mixture distribution consisting of dominating and outlying distributions, and exhibited that as a point estimator of the mean of the dominated distribution, the power mean performs better than the maximum likelihood estimator. 
More recently,  Buri\'c, Elezovi\'c and Mihokovi\'c \cite{Buric2023} presented estimation formulas for the expectations of the power means and the asymptotic expansion of the power means which is applicable in the case of sufficiently large data. 
Some of the other studies on fractional moments are \cite{Khalique1983, Boyarinov2003, Khan2014, Zhang2017, Shoaib2018, Xu2019}. 

However, to our knowledge, the fractional moment has been considered for positive random variables. 
A feature of our estimator is to consider the fractional moment for random variables {\it supported on $\mathbb{R}$} by allowing the fractional powers to take complex numbers. 
In the proof of \cite[Lemma 1]{Mukhopadhyay2021}, it is stated that positive fractional moments of negative values can take complex numbers, however, they assume that the negative values of a random variable are negligible. 
The situation is the same in \cite{Buric2023}. 
We are dealing with the case that the negative values of a random variable are {\it not} negligible at all.

\section{Variance asymptotics}\label{sec:asym-var}

We first give series of notation used in this paper. 

We take the principle branch of the logarithm of complex numbers, specifically, we let 
$$\log(z) := \log r + i\theta, \ \ z = r\exp(i\theta), \ \ r > 0, \theta \in (-\pi, \pi],$$ 
and 
$$ z^p := \exp(p \log(z)), \ z \in \mathbb{C}.$$
Let $z^p = 0$ if $0 < p < 1$ and $z = 0$. 
For $A \subset \mathbb C$, we denote its closure by $\overline{A}$. 
For $z \in \mathbb{C}$, $\textup{Re}(z)$ and $\textup{Im}(z)$ are the real and imaginary parts of $z$ respectively, and $|z|$ denotes the absolute value of $z$, that is, $|z| = \sqrt{\textup{Re}(z)^2+\textup{Im}(z)^2}$. 

We recall an assumption for generator $f$ of the quasi-arithmetic mean in \cite{Akaoka2021-1}. 

\begin{Ass}\label{ass-f}
Let $\alpha \in \mathbb C$ with $\textup{Im}(\alpha) \le 0$. 
Let $U = U_{\alpha}$ be a simply connected domain containing $\overline{\mathbb H} \setminus \{\alpha\}$. 
Let $f : U \to \mathbb C$ be an injective holomorphic  function such that $f\left(\overline{\mathbb H}  \setminus \{\alpha\}\right)$ is convex. 
\end{Ass}
We remark that $\overline{\mathbb H} \setminus \{\alpha\} = \overline{\mathbb H}$ if $\alpha \notin \mathbb{R}$. 
We see that $f^{-1} : f(U) \to U$ is also holomorphic. 

We say that a complex-valued random variable $Y$ is in $L^{r}, \ r > 0$ if $E\left[\left|Y^{r} \right|\right] < +\infty$, and that $Y$ is in $L^{r+}$ if $Y$ is in $L^{s}$ for some $s > r$. 
Let the expectation of a  complex-valued random variable $Y$ be 
\[ E[Y] := E\left[\textup{Re}(Y)\right] + i E\left[\textup{Im}(Y)\right].  \]
Let the variance of a  complex-valued random variable $Y$ be 
\[ \textup{Var}(Y) = E\left[ \left|Y - E[Y]\right|^2 \right]. \]
For real-valued random variables, this definition is equal to the usual definition of variances. 
We see that 
\[ \textup{Var}(Y) = \textup{Var}\left(\textup{Re}(Y)\right) + \textup{Var}\left(\textup{Im}(Y)\right). \]

Throughout this paper, we let $X_1, X_2 \cdots, $ be i.i.d.\ real-valued {\it continuous} random variables. 
Let $1 \le r < +\infty$. 
If $X_1 \in L^r$, then, by Minkowski's inequality, 
$\frac{1}{n} \sum_{i=1}^{n} X_i \in L^r$. 
On the other hand, we can show that  if $X_1 \notin L^r$, then, 
$\frac{1}{n} \sum_{i=1}^{n} X_i \notin L^r$.  

We let 
$$M^f_n :=  f^{-1} \left( \frac{1}{n} \sum_{j=1}^{n} f(X_j) \right).$$
We remark that it can happen that 
$M^f_n \in L^r$ even if $X_1 \notin L^r$.  
It is easy to establish the law of large numbers, the central limit theorem, and the large and moderate deviation principles for $\left(M^f_n\right)_n$.   
The following is our first main result. 

\begin{Thm}\label{thm:gen-var-lim}
Assume that  $E\left[\left|f(X_1)\right|^2\right] < +\infty$, $E[f(X_1)] \ne 0$ and 
\begin{equation}\label{eq:ass-integrable} 
\limsup_{n \to \infty} E\left[ \left| M^f_n \right|^{2+\varepsilon_0} \right] < +\infty 
\end{equation}
for some $\varepsilon_0 > 0$. 
Then, 
\begin{equation}\label{eq:limvar-scaled-gen} 
\lim_{n \to \infty} n \textup{Var}\left( M_n^f \right) = \frac{\textup{Var} (f(X_1))}{\left| f^{\prime} \left( f^{-1}(E[f(X_1)]) \right) \right|^2}.  
\end{equation} 
\end{Thm}

In the above theorem, the assumptions are imposed for the distribution of $f(X_1)$, not for the distribution of $X_1$ itself. 
We call the value of $\dfrac{\textup{Var} (f(X_1))}{\left| f^{\prime} \left( f^{-1}(E[f(X_1)]) \right) \right|^2}$ the {\it asymptotic variance} of $(M^f_n)_n$. 
We can apply this theorem to not only the power means but also the framework of \cite{Akaoka2021-1}, specifically, the geometric and harmonic means. 
See \cite[Theorems 2.1 and 3.1]{Akaoka2021-1}. 
By the assumption that $X_1$ is continuous and the strong law of large numbers, 
\begin{equation}\label{eq:SLLN-cvx} 
E[f(X_1)] = \lim_{n \to \infty} \frac{1}{n} \sum_{j=1}^{n} f(X_j) \in f\left(\overline{\mathbb H}  \setminus \{\alpha\}\right),  \ P\textup{-a.s.} 
\end{equation} 
By the Riemann mapping theorem, $ f^{\prime} \ne 0$ on $U$. 
Furthermore, since $f$ is injective and $X_1$ is continuous, 
$\textup{Var}(f(X_1)) > 0$. 

We identify $\mathbb{C}$ with $\mathbb{R}^2$. 
Let $J (f^{-1})$ be the Jacobi matrix of $f^{-1}$ at $E[f(X_1)]$, 
that is, 
\[ J (f^{-1}) = \begin{pmatrix} \frac{\partial \textup{Re}(f^{-1})}{\partial x}(E[f(X_1)]) & \frac{\partial \textup{Re}(f^{-1})}{\partial y} (E[f(X_1)])\\ \frac{\partial \textup{Im}(f^{-1})}{\partial x} (E[f(X_1)]) & \frac{\partial \textup{Im}(f^{-1})}{\partial y} (E[f(X_1)]) \end{pmatrix}. \]

Let $\textup{Cov}\left(f(X_1)\right)$ be the covariance matrix of the $\mathbb{R}^2$-valued random variable $f(X_1)$, that is, 
\begin{equation*}
\textup{Cov}\left(f(X_1)\right) =  \begin{pmatrix} \textup{Var}(\textup{Re}(f(X_1)))  & \textup{Cov}\left(\textup{Re}(f(X_1)), \textup{Im}(f(X_1))\right) \\ \textup{Cov}\left(\textup{Re}(f(X_1)), \textup{Im}(f(X_1))\right)  & \textup{Var}(\textup{Im}(f(X_1))) \end{pmatrix}. 
\end{equation*}

By the delta method, we can easily establish the central limit theorem for $M^f_n$, specifically, 
\begin{equation}\label{eq:delta-CLT}
 \sqrt{n} \left(  M_n^f - f^{-1}(E[f(X_1)])\right) \Rightarrow N\left(0, J(f^{-1}) \textup{Cov}(f(X_1)) J(f^{-1})^{\prime} \right), \ n \to \infty,   
\end{equation}
where $\Rightarrow$ means the convergence in distribution and  
$N(\cdot, \cdot)$ denotes the two-dimensional normal distribution. 
See \cite[Theorem 1.5]{Akaoka2021-1}. 
We can apply this to construct asymptotic confidence discs for the Cauchy distribution. 
See Section \ref{sec:Cauchy} for more details.  

If we consider the case that $p=1$, that is, $f(x) = x+\alpha$, then, the asymptotic variance is $\textup{Var}(X_1)$ and it holds that $n \textup{Var}\left( M_n^f \right) = \textup{Var}(X_1)$ for every $n$. 
However, as indicated in \cite[Example 10.1.8]{Cassela1990}, it is delicate to handle the asymptotic variances in general.  
Our proof depends on the fast decay of the tails of the normal distribution. 

\begin{proof} 
We first show that 
\begin{equation}\label{eq:intermediate} 
\lim_{n \to \infty} n E \left[ \left| M_n^f  - f^{-1}(E[f(X_1)]) \right|^2 \right] = \frac{\textup{Var} (f(X_1))}{\left| f^{\prime} \left( f^{-1}(E[f(X_1)]) \right) \right|^2}. 
\end{equation} 

By the Taylor expansion of the holomorphic function $f^{-1}$ at $E[f(X_1)]$, 
it holds that 
for every $\varepsilon > 0$, there exists $\delta \in (0, |E[f(X_1)]|)$ such that 
\[ \left| f^{-1}(z) - f^{-1}(E[f(X_1)]) - (f^{-1})^{\prime}(E[f(X_1)])(z - E[f(X_1)]) \right| \le \varepsilon |z - E[f(X_1)]| \]
if $|z - E[f(X_1)]| \le \delta$.

We now identify $\mathbb C$ with $\mathbb R^2$. 
We denote the standard inner product of $x, y \in \mathbb R^2$ by $\braket{x,y}$ and denote Euclidian norm of $x \in \mathbb{R}^2$ by $|x|$. 
Let $F_1 := \textup{Re}(f^{-1})$ and $D_1 := \nabla F_1 (E[f(X_1)]) \in \mathbb{R}^2$. 
Then, $\left| \braket{D_1, z - E[f(X_1)]} \right| \le |D_1| \left| z - E[f(X_1)] \right|$. 

There exists a positive constant $C$ such that for $z \in \mathbb{R}^2$ with $\left|z - E[f(X_1)]\right| \le \delta$, 
\[ \left|(F_1 (z) - F_1 (E[f(X_1)]))^2 - \braket{D_1, z - E[f(X_1)]}^2 \right| \]
\[\le \varepsilon |z - E[f(X_1)]| \left( \varepsilon |z - E[f(X_1)]| + \left|\braket{D_1, z - E[f(X_1)]}\right| \right)\]
\[ \le  C  \left|z - E[f(X_1)]\right|^2 \varepsilon. \]

For ease of notation, let $\displaystyle Z_n :=  \frac{1}{n} \sum_{j=1}^{n} f(X_j)$. 
Then, 
\[ n \biggl|E\left[ (F_1 (Z_n) - F_1 (E[f(X_1)]))^2, \left|Z_n - E[f(X_1)]\right| \le \delta\right] \]
\[- E\left[ \braket{D_1, Z_n - E[f(X_1)]}^2, \, \left|Z_n - E[f(X_1)]\right| \le \delta\right]  \biggr| \]
\[ \le C \textup{Var} (f(X_1)) \varepsilon. \]

We will show that 
\begin{equation}\label{eq:wts-1} 
\lim_{n \to \infty} n E\left[\braket{D_1, Z_n - E[f(X_1)]}^2, \left|Z_n - E[f(X_1)]\right| > \delta\right] = 0. 
\end{equation}
Let $W_n := \sqrt{n} (Z_n - E[f(X_1)])$. 
Then, in order to show \eqref{eq:wts-1}, 
it suffices to show that 
\begin{equation}\label{eq:wts-1-2} 
\lim_{n \to \infty}  E\left[\left| W_n\right|^2, |W_n| > \sqrt{n} \delta\right] = 0. 
\end{equation}

By the assumption that $E\left[ \left|f(X_1)\right|^2\right] < +\infty$, 
we can apply the multidimensional central limit theorem, and we see that $W_n \Rightarrow N(0, \textup{Cov}\left(f(X_1)\right)), \ n \to +\infty$. 
By this, we can show that 
\[ \lim_{M \to +\infty} \left(\limsup_{n \to \infty}  E\left[\left| W_n\right|^2, |W_n| > M \right] \right) = 0. \] 
\eqref{eq:wts-1-2} follows  from this. 

We will show that 
\begin{equation}\label{eq:wts-2}  
\lim_{n \to \infty} n E\left[ (F_1 (Z_n) - F_1 (E[f(X_1)]))^2, \, |Z_n - E[f(X_1)]| > \delta\right] = 0. 
\end{equation}

If $|Z_n - E[f(X_1)]| > \delta$, then $|W_n|^2 / \delta^2  \ge n$ and hence, 
it suffices to show that 
\[ \lim_{n \to \infty}  E\left[ (F_1 (Z_n) - F_1 (E[f(X_1)]))^2 |W_n|^2, \, |W_n| > \sqrt{n}\delta\right] = 0.\]

By the assumption \eqref{eq:ass-integrable} and the H\"older inequality, 
in order to show \eqref{eq:wts-2}, it suffices to show that 
for every $p > 2$,
\[ \lim_{n \to \infty}  E\left[ |W_n|^p, |W_n| > \sqrt{n}\delta\right] = 0.\] 
This follows from the fact that $W_n \Rightarrow N(0, \textup{Cov}\left(f(X_1)\right)), \ n  \to +\infty$, as in the above. 

By \eqref{eq:wts-1} and \eqref{eq:wts-2}, 
we see that 
\[ \limsup_{n \to \infty} n \biggl|E\left[ (F_1 (Z_n) - F_1 (E[f(X_1)]))^2\right] - E\left[\braket{D_1, Z_n - E[f(X_1)]}^2\right]  \biggr| \]
\[ \le C \textup{Var} (f(X_1)) \varepsilon. \]
By letting $\varepsilon \to +0$, 
\[ \limsup_{n \to \infty} n \biggl|E\left[ (F_1 (Z_n) - F_1 (E[f(X_1)]))^2\right] - E\left[\braket{D_1, Z_n - E[f(X_1)]}^2\right]  \biggr| = 0. \]

Let $F_2 := \textup{Im}(f^{-1})$ and $D_2 := \nabla F_2 (E[f(X_1)]) \in \mathbb{R}^2$. 
Then, in the same manner as in the case of $F_1$, 
we can show that 
\[ \limsup_{n \to \infty} n \biggl|E\left[ (F_2 (Z_n) - F_2 (E[f(X_1)]))^2\right] - E\left[\braket{D_2, Z_n - E[f(X_1)]}^2\right]  \biggr| = 0. \]

We remark that 
\[ E\left[\braket{D_1, Z_n - E[f(X_1)]}^2  + \braket{D_2, Z_n - E[f(X_1)]}^2 \right] = E\left[\left| J(f^{-1}) (Z_n - E[f(X_1)])\right|^2\right]. \]

By noting the fact that $f$ and $f^{-1}$ are both holomorphic,  
it holds that 
\[  \lim_{n \to \infty} n E\left[\left| J(f^{-1}) (Z_n - E[f(X_1)])\right|^2\right] = \frac{\textup{Var} (f(X_1))}{\left| f^{\prime} \left( f^{-1}(E[f(X_1)]) \right) \right|^2}. \]

Thus we see \eqref{eq:intermediate}.

Since the function $z \mapsto E \left[ \left| M_n^f  - z \right|^2 \right]$ attains its minimum on $\mathbb R^2$ at $z = E\left[M_n^f \right]$, 
\begin{equation}\label{eq:sup-var-upper} 
\limsup_{n \to \infty} n \textup{Var}\left( M_n^f \right) \le \frac{\textup{Var} (f(X_1))}{\left| f^{\prime} \left( f^{-1}(E[f(X_1)]) \right) \right|^2}.  
\end{equation} 

For ease of notation, 
we let $\widetilde W_n = \left(\widetilde W^{(1)}_{n}, \widetilde W^{(2)}_{n}\right) := \sqrt{n} \left(M_n^f - f^{-1}(E[f(X_1)])\right)$.  
Then, by \eqref{eq:delta-CLT}, $\widetilde W^{(1)}_n \Rightarrow N(0, c_1), n \to +\infty,$ for some $c_1 \ge 0$.  
By \eqref{eq:intermediate}, 
\begin{equation}\label{eq:L2bdd-W1}
\limsup_{n \to \infty} E\left[ \left|\widetilde W^{(1)}_n\right|^2 \right] < +\infty.
\end{equation}

For $M > 0$, 
let $\psi_{M} (x) := \begin{cases} x \ \ \ \ |x| \le M \\ -M \ \ x \le -M \\ M \ \ \ x \ge M \end{cases}$. 
Then, 
$\displaystyle \lim_{n \to \infty} E\left[ \psi_M \left(\widetilde W^{(1)}_n\right)\right] = 0$. 
By \eqref{eq:L2bdd-W1}, 
\[ E\left[ \left| \widetilde W^{(1)}_n - \psi_M \left(\widetilde W^{(1)}_n\right) \right| \right]  \le E\left[ \left|\widetilde W^{(1)}_n\right|, \left|\widetilde W^{(1)}_n\right| > M \right]  = O(M^{-1}). \]
Hence, 
$\displaystyle \lim_{n \to \infty} E\left[ \widetilde W^{(1)}_n\right] = 0$. 
In the same manner, we see that 
$\displaystyle \lim_{n \to \infty} E\left[ \widetilde W^{(2)}_n\right] = 0$. 
Therefore, by using \eqref{eq:delta-CLT}, 
\[ \liminf_{n \to \infty} n \textup{Var}\left( M_n^f \right) = \liminf_{n \to \infty}  \textup{Var}\left( \widetilde W_n \right) = \liminf_{n \to \infty}E\left[ \left| \widetilde W_n \right|^2 \right] \]
\begin{equation}\label{eq:inf-var-lower}  
\ge \frac{\textup{Var} (f(X_1))}{\left| f^{\prime} \left( f^{-1}(E[f(X_1)]) \right) \right|^2}.  
\end{equation}

By \eqref{eq:sup-var-upper} and \eqref{eq:inf-var-lower}, we have the assertion. 
\end{proof}

\begin{Prop}\label{prop:var-infinite}
If $f(X_1) \in L^1$ and $f(X_1) \notin L^2$, then, we have the following:\\
(i) 
\begin{equation}\label{eq:prevar-infinite} 
\lim_{n \to \infty} n E\left[ \left| M^f_n  -  f^{-1}(E[f(X_1)]) \right|^2 \right] = +\infty. 
\end{equation}
and furthermore, 
$\left( \sqrt{n}\left(M^f_n - f^{-1}(E[f(X_1)])\right) \right)_{n \ge 1}$ does not converge to any distribution on $\mathbb{R}^2$.  \\
(ii) If additionally 
\begin{equation}\label{eq:ass-integrable-L1} 
\limsup_{n \to \infty} E\left[ \left| M^f_n \right|^{1+\varepsilon_0} \right] < +\infty 
\end{equation}
for some $\varepsilon_0 > 0$, 
then, 
\begin{equation}\label{eq:var-infinite} 
\lim_{n \to \infty} n \textup{Var}\left( M_n^f \right) = +\infty. 
\end{equation} 
\end{Prop}

\begin{proof}
We show (i). 
Since $f$ is locally Lipschitz around $ f^{-1}(E[f(X_1)])$, 
there exist $c_1, c_2 > 0$ such that 
\[ \left| f^{-1}(z) -  f^{-1}(E[f(X_1)]) \right| \ge c_1 |z -  E[f(X_1)]|, \ z \in B\left(E[f(X_1)], c_2\right). \]
Hence it suffices to show that 
\[ \lim_{n \to \infty} n E\left[ \left| \frac{1}{n} \sum_{i=1}^{n} Z_i  \right|^2, \ \  \left| \frac{1}{n} \sum_{i=1}^{n} Z_i \right| \le c_2 \right] = +\infty,  \]
where we let $Z_i := f(X_i) - E[f(X_i)]$. 
For every $M > 0$, 
\[ n E\left[ \left| \frac{1}{n} \sum_{i=1}^{n} Z_i  \right|^2, \ \  \left| \frac{1}{n} \sum_{i=1}^{n} Z_i \right| \le c_2 \right] \ge M P\left( \frac{M}{\sqrt{n}} \le \left| \frac{1}{n} \sum_{i=1}^{n} Z_i \right| \le c_2 \right). \]
Since $Z_1 \in L^1$, 
\[ \liminf_{n \to \infty} n E\left[ \left| \frac{1}{n} \sum_{i=1}^{n} Z_i  \right|^2, \ \  \left| \frac{1}{n} \sum_{i=1}^{n} Z_i \right| \le c_2 \right] 
\ge M  \liminf_{n \to \infty} P\left( \frac{M}{\sqrt{n}} \le \left| \frac{1}{n} \sum_{i=1}^{n} Z_i \right|  \right). \]
Since $Z_1 \notin L^2$ and $E[Z_1] = 0$, 
we can apply the argument in \cite[Exercise 3.4.3]{Durrett2019} and obtain that 
\[ \liminf_{n \to \infty} P\left( \frac{M}{\sqrt{n}} \le \left| \frac{1}{n} \sum_{i=1}^{n} Z_i \right|  \right) \ge \frac{1}{10}.  \]
Since we can take arbitrarily large $M$, we have \eqref{eq:prevar-infinite}. 

Assume that $\left( \sqrt{n}\left(M^f_n - f^{-1}(E[f(X_1)])\right) \right)_{n \ge 1}$ converges to a distribution on $\mathbb{R}^2$.  
By the delta method and \cite[Exercise 3.4.3]{Durrett2019}, 
we conclude $f(X_1) \in L^2$, which contradicts the assumption. 

We show (ii). 
By $f(X_1) \in L^1$ and \eqref{eq:ass-integrable-L1}, 
$\lim_{n \to \infty} E\left[M^f_n\right] = f^{-1}\left(E[f(X_1)]\right)$.  
Let $Z_i := f(X_i) - f\left(E\left[M^f_n\right]\right)$. 
Then, $\frac{1}{n} \sum_{i=1}^{n} Z_i$ converges to $0$ in probability. 
The rest of the proof goes in the same manner as the above proof. 
\end{proof}

\begin{Rem}
(i) We consider the case that one of the assumptions in Theorem \ref{thm:gen-var-lim} or Proposition \ref{prop:var-infinite} fails. 
Assume that $X_1$ follows the Cauchy distribution and $f$ is a power mean, specifically, $f(x) = (x+\alpha)^p$ for $0 < |p| < 1$ and $\alpha \in \overline{\mathbb{H}}$. 
By Theorem \ref{thm:negative-unbiased} (ii), 
if $p=-1$, $\alpha \in \mathbb{R}$, then, $f(X_1) \notin L^1$ and $M_n^f \notin L^1$.  
By Theorem \ref{thm:negative-Cauchy-var} (iii), 
if $p \in (-1,-1/2)$, $\alpha \in \mathbb{R}$, then, $f(X_1) \in L^1$, $f(X_1) \notin L^2$ and \eqref{eq:ass-integrable} holds for some $\varepsilon_0 > 0$.  
Furthermore, \eqref{eq:limvar-scaled-gen} fails and \eqref{eq:var-infinite} holds. 
By Theorem \ref{thm:truncated-integrability}, 
if $p \in (0,1/2)$, then, $f(X_1) \in L^2$ and $M_n^f \notin L^1$. 
Hence \eqref{eq:ass-integrable} fails for every $\varepsilon_0 > 0$ and \eqref{eq:limvar-scaled-gen} fails. 
In this case, we see that $E\left[M^f_n\right] = f^{-1}\left(E[f(X_1)]\right)$, hence we only need to show \eqref{eq:intermediate}. \\
(ii) We can consider the Edgeworth expansion of $\left(M_n^f\right)_n$. 
By \cite[Theorem 2 and Remark 1.1]{Bhattacharya1978}, 
we see that if $f(X_1) \in L^r, r \ge 3$, $J(f^{-1}) \textup{Cov}(f(X_1)) J(f^{-1})^{\prime}$ is nonsingular, and the Cram\'er condition 
$$\limsup_{|(\lambda_1, \lambda_2)| \to +\infty} E\left[\exp\left(i   \left( \lambda_1 \textup{Re}(f(X_1)) + \lambda_2 \textup{Im}(f(X_1)) \right)  \right) \right] = 0$$ 
hold, 
then, there exist smooth integrable functions $(\psi_{s,n})_n$ on $\mathbb R^2$ such that 
\[ \sup_{B: \textup{ open ball in } \mathbb{R}^2} \left|P\left( \sqrt{n} \left(  M_n^f - f^{-1}(E[f(X_1)])\right) \in B\right) - \int_B \psi_{s,n} (x) dx \right| = o\left(n^{-(r-2)/2} \right). \]
\end{Rem}

\section{Negative power means of random variables}\label{sec:neg}

We first deal with the integrability of the negative power means, specifically, $f(x) = (x+\alpha)^p$ for $-1 \le p  < 0$ and $\alpha \in \overline{\mathbb{H}}$. 
We recall that $X_1, X_2 \cdots, $ are i.i.d.\ real-valued continuous random variables. 

\begin{Thm}[integrability]\label{thm:negative-var}
For $r > 0$ and $n \ge 1$, we have the following:\\
(i) 
Let $p \in (-1,0)$ and $\alpha \in \overline{\mathbb H}$. 
Let $X_1 \in L^{r/n}$. 
Then, $M^f_n \in L^r$.\\
(ii) 
Let $p=-1$ and $\alpha \in \mathbb H$. 
Let $X_1 \in L^{2r/n}$. 
Then,  $M^f_n \in L^r$. 
\end{Thm}

The case that $p=-1$ and $\alpha \in \mathbb{R}$ is harder to deal with, because $M^f_n$ is equal to $X_1$ in distribution when $X_1$ follows the Cauchy distribution.

\begin{Thm}\label{thm:neg-var-lim}
Assume that $X_1 \in L^{0+}$. 
Then, \\
(i) Assume that either (a) $p \in [-1,0)$ and $\alpha \in \mathbb H$, 
or, (b) $p \in (-1,0)$, $\alpha \in \mathbb R$ and $f(X_1) \in L^{2}$  
holds. 
Then, \eqref{eq:limvar-scaled-gen} holds.\\
(ii) If $p \in (-1,0)$, $\alpha \in \mathbb R$, $f(X_1) \in L^1$ and $f(X_1) \notin L^{2}$, 
then, \eqref{eq:var-infinite} holds. 
\end{Thm}

This is an extension of \cite[Theorem 3.1]{Akaoka2021-1}. 
By Theorem \ref{thm:negative-var}, $\textup{Var}\left(M^f_n \right) < +\infty$ for sufficiently large $n$ if $X_1 \in L^{0+}$. 

Now we consider limit behaviors of the power means as $p$ tends to $0$ or $1$. 
We let 
\[ G^{(\alpha)}_n := \prod_{j=1}^{n} (X_j + \alpha)^{1/n}  - \alpha, \ \alpha \in \overline{\mathbb{H}}. \]
This is the quasi-arithmetic means of $X_1, \dots, X_n$ with generator $f(x) = \log(x+\alpha)$ and corresponds to the case that $p=0$. 

\begin{Prop}[geometric and harmonic means as parameter limits]\label{prop:negative-parameterlimit}
We see the following claims:\\
(i) Let $\alpha \in \overline{\mathbb H}$.
Let $n \ge 2$. 
Assume that $X_1 \in L^{1/n}$. 
Then, 
\[ \lim_{p \to -0} M^{ f^{(\alpha)}_p}_n = G^{(\alpha)}_n,     
\ \ \textup{ a.s. and  in } L^1.  \]
(ii) Let $\alpha \in \mathbb H$.
Let $n \ge 3$. 
Assume that $X_1 \in L^{2/3+}$. 
Then, 
\[ \lim_{p \to -1+0} M^{ f^{(\alpha)}_p}_n = M^{ f^{(\alpha)}_{-1}}_n,     
\ \ \textup{ a.s. and  in } L^1.  \]
\end{Prop}

\begin{Rem}\label{rmk:harmonic}
(i) In Theorem \ref{thm:negative-var} (iii) and Proposition \ref{prop:negative-parameterlimit} (ii) above, 
we need to assume that $\alpha \notin \mathbb R$ in general. 
Assume that $p=-1$. 
If $X_1$ follows the Cauchy distribution, then, $M^f_n$ also follows the same Cauchy distribution for every $n$.  
Furthermore, it is known that $\left(M^f_n\right)_n$ converges weakly to the Cauchy distribution as $n \to \infty$ 
if on an open interval $I$ containing $0$, $X_1$ has a density function $f_X$ which is H\"{o}lder continuous and positive on $I$.
See \cite[Example 3.8.4]{Durrett2019} or \cite[Theorem XVII.5.3]{Feller1966} for more details.
\cite{Lehmann1988} also gave related discussions.\\
(ii) Assume that $p=-1$ and $\alpha = 0$. 
If $X_1$ follows the standard {\it log-Cauchy distribution}, that is, the distribution of $X_1$ is equal to $\exp(Y)$ where $Y$ follows the standard Cauchy distribution, 
then, $X_1 \notin L^{0+}$ and furthermore $\min\{X_1, \cdots, X_n\} \notin L^{0+}$ for every $n \ge 1$. 
We remark that $M^f_n \ge \min\{X_1, \cdots, X_n\}$. 
We also see that $-\log \min\{X_1, \cdots, X_n\} / n \Rightarrow G, \ n \to \infty$, where $G$ is the distribution on $(0,\infty)$ with density function 
\[ g(x) = \frac{1}{x^2} \exp\left(-\frac{1}{x}\right), \ x > 0.  \]
We are not sure whether $M^f_n \notin L^{0+}$ or not for $\alpha \in \mathbb H$.\\
(iii) Under our assumption for $(p, \alpha)$, 
we see that for every $z_1, \cdots, z_n \in \mathbb{R}$, 
$$\sum_{j=1}^{n} (z_j + \alpha)^{p} \notin (-\infty, 0). $$ 
\end{Rem}

Before we proceed to the proofs, we check the generator $f$ satisfies Assumption \ref{ass-f}. 
If $p = -1$ and $\alpha \in \mathbb H$\footnote{We cannot let $\alpha \in \mathbb{R}$. See \cite[Example 1.2 (ii)]{Akaoka2021-1}.}, then, it is shown in \cite[Example 1.2 (ii)]{Akaoka2021-1}. 
If $p=0$ and $\alpha \in \overline{\mathbb H}$, then it is shown in \cite[Lemma 2.4]{Akaoka2021-1}. 

\begin{Prop}\label{prop:ass-neg}
Let $p \in (-1,0)$ and $\alpha \in \overline{\mathbb H}$. 
Let $f(z) := (z+\alpha)^p$  for $z \in \overline{\mathbb{H}} \setminus\{-\alpha\}$. 
Then, 
$f(\overline{\mathbb H} \setminus\{-\alpha\})$ is convex. 
In particular, for every $n \ge 1$ and $z_1, \dots, z_n \in \overline{\mathbb H}  \setminus\{-\alpha\}$, 
\[ f^{-1}\left(\frac{1}{n} \sum_{i=1}^{n} f(z_i) \right) \in \overline{\mathbb H} \setminus\{-\alpha\}. \] 
\end{Prop}

We remark that $\overline{\mathbb H} \setminus\{-\alpha\} = \overline{\mathbb H}$ if $\alpha \notin \mathbb{R}$. 

\begin{proof}
Let $\alpha \in \mathbb R$. 
Then, 
$$f(\overline{\mathbb H} \setminus\{-\alpha\}) = \left\{r\exp(i\theta) : r > 0, \theta \in [p \pi, 0] \right\}.$$ 
Since $0 > p > -1$, this set is convex. 

Let $\alpha \in \mathbb H$. 
We can assume that $\textup{Im}(\alpha) = 1$ without loss of generality. 
Then, 
\[ f(\mathbb R) = \left\{\exp(i p \theta) \sin^{-p}(\theta) : \theta \in (0, \pi) \right\}\] 
and $f(\overline{\mathbb H})$ is a bounded set surrounding by $f(\mathbb R)$. 
Let $\gamma(\theta) := \sin^{-p}(\theta) (\cos(p\theta), \sin(p\theta))$ and $\nu(\theta) := (-\sin((p-1)\theta), \cos((p-1)\theta))$. 
The pair $(\gamma, \nu)$ is called a Legendre curve, that is,  $\gamma^{\prime}(\theta) \cdot \nu(\theta) = 0$. 
See \cite{Fukunaga2016} for the definition.  
and we can apply \cite[Theorem 1.6]{Fukunaga2016} to $(\gamma, \nu)$, and we have the assertion. 
\end{proof}

Now we proceed to the proofs of Theorems \ref{thm:negative-var} and \ref{thm:neg-var-lim} and Proposition \ref{prop:negative-parameterlimit}. 
Hereafter, for ease of notation, we often let $Y_j := X_j + \alpha$. 
Assume that  $r_j > 0$ and $\theta_j \in (0, \pi)$ satisfy that $r_j \exp(i \theta_j) = Y_j$. 
In the following proofs, $C_{p,1}, C_{p,2}, \cdots$ are positive constants depending only on $p$.

\begin{proof}[Proof of Theorem \ref{thm:negative-var}]
(i) 
Let $\varepsilon := \dfrac{-(1+p)\pi}{4p} = \dfrac{(1-|p|)\pi}{4|p|} > 0.$
Then, 
\begin{equation}\label{eq:rotation} 
\left|\frac{1}{n} \sum_{j=1}^{n} Y_j^{p} \right| = \left|\frac{1}{n} \sum_{j=1}^{n} \exp(i\varepsilon p) Y_j^{p} \right|. 
\end{equation} 
Since $\exp(i\varepsilon p) Y_j^{p} = r_j^p \exp(i(\varepsilon + \theta_j)p)$
and $(\varepsilon + \theta_j)|p| \in \left[(1+p)\pi/4, (1-3p)\pi/4\right],$
we see that 
\[ \textup{Im}\left( \exp(i\varepsilon p) Y_j^{p} \right) \ge  C_{p,1} |Y_j|^p. \]
By this and \eqref{eq:rotation}, 
\[ \left|\frac{1}{n} \sum_{j=1}^{n} Y_j^{p} \right|^{r/p} \le  C_{p,2} \left|\frac{1}{n} \sum_{j=1}^{n} |Y_j|^{p} \right|^{r/p}. \]
By this and the geometric mean-harmonic mean inequality, 
\begin{equation}\label{eq:GH1}
\left|\frac{1}{n} \sum_{j=1}^{n} Y_j^{p} \right|^{r/p} \le C_{p,3} \prod_{j=1}^{n} |Y_j|^{r/n}. 
\end{equation}
Assertion (i) follows from this inequality.  

(ii) 
We remark that
\[ \left| \frac{1}{n} \sum_{j=1}^{n} Y_j^{-1} \right| = \dfrac{n}{\left| \sum_{j=1}^{n} \frac{\overline{Y_j}}{|Y_j|^2}\right|} \le \frac{n}{\textup{Im}(\alpha) \sum_{j=1}^{n} |Y_j|^{-2}}.  \]
By this and the geometric mean-harmonic mean inequality, 
we see that 
\begin{equation}\label{eq:GH2}
 E\left[ \left| \frac{1}{n} \sum_{j=1}^{n} Y_j^{-1} \right|^{-r} \right] 
\le \frac{1}{\textup{Im}(\alpha)^r} E\left[ \left( \frac{n}{\sum_{j=1}^{n} |Y_j|^{-2}} \right)^r  \right] 
\le \frac{1}{\textup{Im}(\alpha)^r} E\left[ |Y_j|^{2r/n} \right]^n < +\infty. 
\end{equation}
\end{proof}

\begin{proof}[Proof of Theorem \ref{thm:neg-var-lim}]
We show (i). 
We will apply Theorem \ref{thm:gen-var-lim}. 
Since $f(-\alpha) = 0$,  
we see that $0 \notin f\left(\overline{\mathbb H}  \setminus \{\alpha\}\right)$. 
Since $X_1$ is continuous, $P(X_1 \ne \alpha) = 1$. 
By \eqref{eq:SLLN-cvx}, if $f(X_1) \in L^1$, which is shown later, then, $E[f(X_1)] \in f\left(\overline{\mathbb H}  \setminus \{\alpha\}\right)$. 
Hence $E[f(X_1)] \ne 0$. 

If $\alpha \in \mathbb H$, then, $f(X_1) = (X_1 + \alpha)^p$ is bounded and hence is in $L^2$, 
regardless of any integrability assumptions of $X_1$. 
If $p > -1$, $\alpha \in \mathbb R$ and $(X_1 + \alpha)^{-1} \in L^{-1/2p}$. 
Then, $f(X_1) = (X_1 + \alpha)^p \in L^2$.  

Since $X_1 \in L^{0+}$, by using l'Hospital's theorem, 
\begin{equation}\label{eq:LHospital} 
\lim_{\varepsilon \to +0} E[|Y_1|^{\varepsilon}]^{1/\varepsilon} = \exp\left(E \left[\log |Y_1| \right]\right) < +\infty. 
\end{equation}

Then \eqref{eq:ass-integrable} follows from \eqref{eq:GH1} and \eqref{eq:LHospital} if $p = -1$, and \eqref{eq:GH2} and \eqref{eq:LHospital} if $p > -1$. 

We can show (ii) by applying Proposition \ref{prop:var-infinite}. 
We have \eqref{eq:ass-integrable-L1} by \eqref{eq:LHospital}. 
The rest of the proof goes in the same manner as in the above proof. 
\end{proof}

\begin{proof}[Proof of Proposition \ref{prop:negative-parameterlimit}] 
(i) We see that $Y_j \ne 0$ for every $j$ almost surely. 
Then, 
\[ \lim_{p \to -0} \left(\frac{1}{n} \sum_{j=1}^{n} Y_j^{p} \right)^{\frac{1}{p}} = \prod_{j=1}^{n} Y_j^{1/n}, \ \ \textup{ a.s.} \]
It holds that 
\[ \cos(p \pi) \sum_{j=1}^{n} \left|Y_j \right|^{p}  \le \left| \sum_{j=1}^{n} Y_j^{p} \right|.  \]
By this and the geometric mean-harmonic mean inequality, 
\[  \left| \frac{1}{n} \sum_{j=1}^{n} Y_j^{p} \right|^{\frac{1}{p}} \le (\cos(p \pi))^{\frac{1}{p}} \prod_{j=1}^{n} |Y_j|^{1/n}.  \]
By this, the assumption that $X_1 \in L^{1/n}$ and 
$\lim_{p \to -0} (\cos(p \pi))^{\frac{1}{p}}  = 1, $
we can apply the dominated convergence theorem and obtain (i). 

(ii) Since $\textup{Im}(Y_j^{p}) < 0$, 
we obtain that by the inequality of the geometric mean and the harmonic mean, 
\[ E\left[ \left| \frac{1}{n} \sum_{j=1}^{n} Y_j^{p} \right|^{\frac{1}{p}} \right] \le E\left[  \left| \frac{1}{n} \sum_{j=1}^{n} -\textup{Im}\left(Y_j^{p}\right) \right|^{\frac{1}{p}} \right] \le E\left[ \left(-\textup{Im}\left( Y_1^{p} \right) \right)^{\frac{1}{np}}\right]^n.  \] 
Since $\textup{Im}(Y_1) = \textup{Im}(\alpha) = r_1 \sin\theta_1$,  
we see that 
 \[ \left(-\textup{Im}\left( Y_1^{p} \right) \right)^{\frac{1}{p}} 
 = r_1 (\sin(-p\theta_1))^{\frac{1}{p}} 
 = \textup{Im}(\alpha)^{\frac{1}{p}} r_1^{1 - 1/p} \left(\frac{\sin \theta_1}{\sin(-p\theta_1)}\right)^{\frac{1}{-p}}. \] 
 We also obtain that 
 \[ \sup_{-1 < p < -7/8, \theta \in (4\pi/7, \pi)} \left(\frac{\sin \theta}{\sin(-p\theta)}\right)^{\frac{1}{-p}} \le 1, \]
 and, 
 \[ \sup_{-1 < p < -7/8, \theta \in (0, 4\pi/7]} \left(\frac{\sin \theta}{\sin(-p\theta)}\right)^{\frac{1}{-p}} < +\infty.\] 
 Hence, there exists a constant $C$ independent from $p$ such that 
 \[ \sup_{-1 < p < -7/8} \left(-\textup{Im}\left( Y_1^{p} \right) \right)^{\frac{1}{p}} \le C \left|  Y_1 \right|^{1-1/p}. \]
 By this and the assumptions that $n \ge 3$, $p > -1$, and $X_1 \in L^{2/3+}$, 
the family $\left\{  \left| \frac{1}{n} \sum_{j=1}^{n} Y_j^{p} \right|^{\frac{1}{p}}\right\}_{p \in (-1,0)}$ is uniformly integrable.  
By this and $$\lim_{p \to -1+0} \left( \frac{1}{n} \sum_{j=1}^{n} Y_j^{p} \right)^{\frac{1}{p}} = \frac{n}{\sum_{j=1}^{n} Y_j^{-1}},$$
we obtain an assertion (ii). 
\end{proof}

\section{Positive power means of random variables}\label{sec:posi}

We consider the case that $f(x) = (x+\alpha)^p$ for $0 < p < 1$ and $\alpha \in \overline{\mathbb{H}}$.  
Contrary to the negative power means, integrability is hard to be assured. 
We recall that $X_1, X_2 \cdots, $ be i.i.d.\ real-valued continuous random variables. 

\begin{Prop}[integrability]\label{prop:positive-integrability}
It holds that\\
(i) If $X_1 \in L^1$, then, $M^{f^{(\alpha)}_p}_n \in L^1$ for every $\alpha \in \overline{\mathbb H}$ and $n \ge 1$. \\
(ii) If $X_1 \notin L^1$, then, $M^{f^{(\alpha)}_p}_n \notin L^1$ for every $\alpha \in \overline{\mathbb H}$ and $n \ge 1$.
\end{Prop}

Therefore, it is natural to consider some truncations for positive power means of random variables,  in order to make them integrable when $X_1 \notin L^1$. 

\begin{Thm}[integrability of truncated sums]\label{thm:truncated-integrability}
Let $n \ge 1$ and $X_1 \in L^{\max\{p,1-p\}}$. 
Let $\alpha \in \overline{\mathbb H}$.
Then,  
\begin{equation}\label{eq:trancated-L1}
n^{1/p} M^{f^{(\alpha)}_p}_n - n M^{f^{(\alpha)}_1}_n \in L^1.
\end{equation}  
\end{Thm}

The following corresponds to Proposition \ref{prop:negative-parameterlimit} (i). 

\begin{Thm}[geometric mean as a parameter limit of truncated sum]\label{thm:positive-parameterlimit}
Let $n \ge 2$. 
Let $\alpha \in \overline{\mathbb H}$. 
Assume that 
\begin{equation}\label{ass-fine}
E\left[ \left|f_{1-\eta}^{(\alpha)}(X_1) \right| \right] = E\left[ \left|X_1 + \alpha\right|^{1-\eta}\right] = O\left(\eta^{-\ell}\right), \ \ \eta \to +0 
\end{equation}
for some $\ell > 0$. 
Then, 
\begin{equation}\label{eq:positive-parameterlimit-gen}
\lim_{p \to +0}  \frac{n^{1/p} M^{f^{(\alpha)}_p}_n - n M^{f^{(\alpha)}_1}_n}{n^{1/p} - n} = G_n^{(\alpha)}, \ \ \textup{ a.s. and in $L^1$}.  
\end{equation}
\end{Thm}

The assumption \eqref{ass-fine} is satisfied by the Cauchy distribution.  
See Section \ref{sec:Cauchy} for more details.

Before we proceed to the proofs, we check the generator $f$ satisfies Assumption \ref{ass-f}. 

\begin{Prop}\label{prop:ass-posi} 
Let $p \in (0,1)$ and $\alpha \in \overline{\mathbb H}$. 
Let $f(z) := (z+\alpha)^p$ for $z \in \overline{\mathbb{H}}$. 
Then, 
$f(\overline{\mathbb H})$ is convex. 
In particular, for every $n \ge 1$ and $z_1, \dots, z_n \in \overline{\mathbb H}$, 
\[ f^{-1}\left(\frac{1}{n} \sum_{i=1}^{n} f(z_i) \right) \in \overline{\mathbb H}. \]
\end{Prop}

We recall that it is assumed that $z^p = 0$ if $0 \le p \le 1$ and $z = 0$. 

\begin{proof}
Let $\alpha \in \mathbb R$. 
Then, 
$$f(\overline{\mathbb H}) = \left\{r\exp(i\theta) : r \ge 0, \theta \in [0, p\pi] \right\}.$$
Since $0 < p < 1$, this set is convex. 

Assume that $\alpha \in \mathbb{H}$. 
Let 
$$D_s := \left\{t \exp(ip\theta) \sin^{-p}\theta | \theta \in (0, \pi), t \ge s \right\}$$ 
for $s \ge 0$. 
Then, 
$f(\overline{\mathbb H}) = D_{\textup{Im}(\alpha)^p}$. 
Hence it suffices to show that $D_s$ is convex for every $s > 0$. 
Since $D_s = sD_1$, it suffices to deal with the case that $s=1$. 

For $\theta \in (0, \pi)$, 
\[ H_{\theta} := \left\{(x,y) \middle| \sin((1-p)\theta)x+ \cos((1-p)\theta)y \ge \sin^{1-p} \theta \right\}.  \]
Now it suffices to show that $D_1 = \cap_{\theta \in (0,\pi)} H_{\theta}$.  

We first show that $D_1 \subset H_{\theta}$ for every $\theta \in (0,\pi)$. 
Let $t_0 \ge 1$ and $\theta_0 \in (0,\pi)$. 
Then, by the angle sum formula for trigonometric functions, $t_0 \exp(ip\theta_0) \in H_{\theta}$ if and only if $t_0 \sin((1-p)\theta + p\theta_0) \ge \sin^{1-p}(\theta)  \sin^p (\theta_0)$. 
The latter inequality follows from $t_0 \ge 1$ and the fact that $\log(\sin(\xi))$ is concave\footnote{This was also used in the proof of \cite[Lemma 2.4]{Akaoka2021-1}.} as a function of $\xi$. 

We second show that $\cap_{\theta \in (0,\pi)} H_{\theta} \subset D_1$. 
Let $(x,y) \in \cap_{\theta \in (0,\pi)} H_{\theta}$. 
By considering the cases that $\theta$ is sufficiently close to $0$ and $\pi$ respectively, 
$x+yi = t_1 \exp(ip\theta_1) \in \cup_{s > 0} D_s$ for some $t_1 > 0$ and $\theta_1 \in (0,\pi)$. 
Since $t_1 \sin((1-p)\theta + p\theta_1) \ge \sin^{1-p}(\theta)  \sin^p (\theta_1)$ holds  for $\theta \in (0,\pi)$, in particular, $\theta = \theta_1$, 
we see that $t_1 \ge 1$. 

Thus we see  that $D_1 = \cap_{\theta \in (0,\pi)} H_{\theta}$.  
\end{proof}

For ease of notation, we let 
$S_{p,n} := \sum_{j=1}^{n} Y_j^{p}$. 

\begin{proof}[Proof of Proposition \ref{prop:positive-integrability}]
Assertion (i) is an easy consequence of the H\"older inequality. 

We show (ii). 
We first consider the case that $E\left[|Y_1|, \ X_1 + \textup{Re}(\alpha) \ge 0\right] = +\infty$. 
Then, we obtain that if $X_j + \textup{Re}(\alpha) \ge 0$, then, 
$$\textup{Re}\left( Y_j^{p} \right) = \left| Y_j \right|^{p} \cos\left(\frac{p\pi}{2}\right) \ge 0.$$
We also see that $P\left(X_1 + \textup{Re}(\alpha) \ge 0\right) > 0$. 
Hence, 
\[ E\left[ \left|S_{p,n} \right|^{\frac{1}{p}} \right] \ge P\left(X_1 + \textup{Re}(\alpha) \ge 0\right)^{n-1} \cos^{\frac{1}{p}}\left(\frac{p\pi}{2}\right)  E\left[ \left|Y_1 \right|, \ X_1 + \textup{Re}(\alpha) \ge 0\right] = +\infty.  \]

We second consider the case that $E\left[|Y_1|, \ X_1 + \textup{Re}(\alpha) < 0\right] = +\infty$. 
$X_j + \textup{Re}(\alpha) < 0$ if and only if $\theta_j > \pi/2$, and, 
$$\textup{Im}\left( Y_j^p \right) = r_j^p \sin(p \theta_j) \ge 0. $$
Hence, 
\[ E\left[ \left|S_{p,n} \right|^{\frac{1}{p}} \right] \ge E\left[ \left|\sum_{j=1}^{n} |Y_j|^{p} \sin(p \theta_j)\right|^{\frac{1}{p}} \right] \ge E\left[ |Y_1| \sin^{\frac{1}{p}}(p \theta_1), \ X_1 + \textup{Re}(\alpha) < 0 \right] \]
\[= E\left[ |Y_1| \sin^{\frac{1}{p}}(p \theta_1), \ \theta_1 > \pi/2 \right] \ge \inf_{\theta \in [\pi/2, \pi)} \sin^{\frac{1}{p}}(p \theta) E\left[ |Y_1|, \ X_1 + \textup{Re}(\alpha) < 0 \right] = +\infty. \]
\end{proof}

\begin{proof}[Proof of Theorem \ref{thm:truncated-integrability}]

We show this by induction in $n$. 
The assertion is obvious for $n=1$. 
Assume that the assertion holds for $n=m$. 
Now we show the assertion holds for $n=m+1$. 
 
\begin{Lem}\label{lem:positive-fundamental}
For every $p \in (0,1)$, there exists a positive constant $C_p$ such that 
for every $x, y \in \overline{\mathbb H}$, 
\[ \left|(x^p + y^p)^{\frac{1}{p}} - (x+y)\right| \le C_p \max\left\{|x|^p |y|^{1-p}, |x|^{1-p} |y|^p \right\}. \]
\end{Lem}

We assume this lemma. 
We see that 
\[ \left| S_{p, m+1}^{\frac{1}{p}} - S_{1, m+1}\right|  \le \left| S_{p, m+1}^{\frac{1}{p}} - \left( S_{p, m}^{\frac{1}{p}} + Y_{m+1} \right) \right|  +  \left| S_{p, m}^{\frac{1}{p}} - S_{1,m} \right|. \]

Then, 
by applying Lemma \ref{lem:positive-fundamental} to the case that $x = S_{p,m}^{\frac{1}{p}}$ and $y = Y_{m+1}$, 
we see that 
\[ E\left[ \left|S_{p, m+1}^{\frac{1}{p}} - \left( S_{p,m}^{\frac{1}{p}} + Y_{m+1} \right) \right| \right] \]
\[\le C_p \left(E\left[\left|S_{p,m}\right|\right] E\left[|Y_{m+1}|^{1-p}\right] + E\left[\left| S_{p,m} \right|^{\frac{1}{p} -1}\right] E\left[|Y_{m+1}|^p\right] \right).\]
By the assumption, 
$E\left[|Y_{m+1}|^{1-p} + |Y_{m+1}|^{p}\right] < +\infty$ and $E\left[\left|S_{p,m}\right|\right] \le m E\left[|Y_{1}|^{p}\right] < +\infty$.
By the H\"older inequality, 
\[ E\left[\left| S_{p,m}  \right|^{\frac{1}{p} -1}\right] \le  m^{\frac{1}{p} - 1} E\left[\left( \frac{1}{m} \sum_{j=1}^{m} |Y_j|  \right)^{1 -p}\right] \le m^{p + \frac{1}{p} - 1} E\left[|Y_1|^{1-p}\right]  < +\infty.\] 
By this and the assumption for induction,  
we see the assertion holds also for $n=m+1$. 
\end{proof}

\begin{proof}[Proof of Lemma  \ref{lem:positive-fundamental}]
If $x=0$ or $y=0$, then, the assertion obviously holds. 
Hereafter we assume that $xy \ne 0$. 
By symmetry and a suitable clockwise rotation of $x$ and $y$, we can assume that $x > 0$ and $y \in \overline{\mathbb H}$. 
By scaling, we can further assume that $x=1$ and $y \in \overline{\mathbb H}$. 
Now it suffices to show that there exists $C_p$ such that for every $y \in \overline{\mathbb H}$, 
\begin{equation}\label{eq:reduction-sufficient}
\left|(1+y^p)^{1/p} - 1 - y \right| \le C_p \max\left\{|y|^p, |y|^{1-p}\right\}
\end{equation}

We show this by considering two cases according to the value of $|y|$. 

We first consider the case that $|y| \le 1$. 
We see that 
\[ \lim_{y \to 0, y \in \overline{\mathbb{H}} \setminus \{0\}} \frac{(1+y^p)^{\frac{1}{p}} - 1 - y}{y^p} =\frac{1}{p}. \]
Furthermore, as a function of $y$, 
$\frac{(1+y^p)^{\frac{1}{p}} - 1 - y}{y^p}$ is holomorphic on $\mathbb D \cap \mathbb H$, where we let $\mathbb D := \left\{z \in \mathbb C : |z| < 1 \right\}$.  
Hence, by the maximal principle, there exists a positive constant $C_{p,1}$ such that 
\[ \left| (1+y^p)^{\frac{1}{p}} - 1 - y \right| \le C_{p,1} |y|^p, \  \ y \in \overline{\mathbb H}, \ |y| \le 1, \] 
and hence \eqref{eq:reduction-sufficient} holds.

We second consider the case that $|y| \ge 1$. 
Let $z := \frac{1}{y^p}$ for $y$ such that $|y| \ge 1$. 
Since $p < 1$, it holds that $z \in (-\mathbb H) \cup (0, \infty)$ and in particular $z \notin (-\infty,0)$.  
We see that 
\[ \frac{(1+y^p)^{\frac{1}{p}} - 1 - y}{y^{1-p}} = z^{\frac{1}{p} - 1} \left(  \left(1+\frac{1}{z} \right)^{\frac{1}{p}} - 1 - \left(\frac{1}{z}\right)^{\frac{1}{p}} \right). \]
Since $\frac{1}{z}, 1+\frac{1}{z} \in \mathbb H \cup (0, \infty)$, 
\[ z^{\frac{1}{p}} \left(1+\frac{1}{z} \right)^{\frac{1}{p}} = (z+1)^{\frac{1}{p}}, \textup{   and,   }  z^{\frac{1}{p}} \left(\frac{1}{z} \right)^{\frac{1}{p}} = 1. \]
Hence, 
\begin{equation}\label{eq:transfer} 
\frac{(1+y^p)^{\frac{1}{p}} - 1 - y}{y^{1-p}} = \frac{(1+z)^{\frac{1}{p}} - 1 - z^{\frac{1}{p}}}{z}.
\end{equation}
Hence, 
\[ \lim_{|y| \to \infty, \ y \in \overline{\mathbb H}} \frac{(1+y^p)^{\frac{1}{p}} - 1 - y}{y^{1-p}} = \lim_{z \to 0} \frac{(1+z)^{\frac{1}{p}} - 1 - z^{\frac{1}{p}}}{z} = \frac{1}{p}. \]
Furthermore, a function of $y$, 
$\frac{(1+y^p)^{\frac{1}{p}} - 1 - y}{y^{1-p}}$ is continuous on $\overline{\mathbb H} \setminus \mathbb D$. 
Hence, by the maximal principle, 
there exists a positive constant $C_{p,2}$ such that 
\[ \left| (1+y^p)^{\frac{1}{p}} - 1 - y \right| \le C_{p,2} |y|^{1-p},  \ \ y \in \overline{\mathbb H}, \ |y| \ge 1.\]  
Hence \eqref{eq:reduction-sufficient} holds if we let $C_p := \max\{C_{p,1}, C_{p,2}\}$. 
\end{proof}

\begin{Rem}\label{rem:unif-bdd}
For $p \in (0,1)$,  
let $\mathbb D_{p, \pm} := \left\{z^{\pm p} :z \in  \overline{\mathbb H \cap \mathbb D} \right\}$. 
We remark that $-1 \notin D_{p, \pm} \subset \overline{\mathbb D}$.   
Let $0 < p_1 \le p_2 < 1$. 
Let $C_p, C_{p,1}, C_{p,2}$ be the constants in the proof of Lemma \ref{lem:positive-fundamental}. 

Then, as a function of $(p,y)$, 
$\frac{(1+y)^{\frac{1}{p}} - 1}{y}$ is continuous on $(p,y) \in [p_1, p_2] \times (\mathbb D_{p, +} \cap \{z : |z| \ge 1/2\})$.
Hence, 
\[ \sup_{p \in [p_1, p_2]} \sup_{y \in \mathbb D_{p, +} \cap \{z : |z| \ge 1/2\}} \left| \frac{(1+y)^{\frac{1}{p}} - 1 - y^{\frac{1}{p}}}{y} \right| < +\infty. \]

Let $\frac{(1+y)^{\frac{1}{p}} - 1}{y} := \frac{1}{p}$ if $y=0$. 
By the Taylor expansion of $(1+y)^{\frac{1}{p}}$ at $y = 0$, 
\[ \sup_{p \in [p_1, p_2]} \sup_{y \in \{z : |z| \le 1/2\}} \left| \frac{(1+y)^{\frac{1}{p}} - 1}{y} - \frac{1}{p} \right| < +\infty. \]
Hence, 
\[ \sup_{p \in [p_1, p_2]} \sup_{y \in \{z : |z| \le 1/2\}} \left| \frac{(1+y)^{\frac{1}{p}} - 1 - y^{\frac{1}{p}}}{y} \right| < +\infty. \]
Hence, 
\[ \sup_{p \in [p_1, p_2]} \sup_{y \in \mathbb D_{p, +}} \left| \frac{(1+y)^{\frac{1}{p}} - 1 - y^{\frac{1}{p}}}{y} \right| < +\infty. \]
Hence $$\sup_{p \in [p_1, p_2]} C_{p,1} < +\infty.$$

In the same manner, we see that 
\[ \sup_{p \in [p_1, p_2]} \sup_{y \in \mathbb D_{p, -}} \left| \frac{(1+y)^{\frac{1}{p}} - 1 - y^{\frac{1}{p}}}{y} \right| < +\infty. \]
By this and \eqref{eq:transfer},
$$\sup_{p \in [p_1, p_2]} C_{p,2} < +\infty.$$

Hence, for $0 < p_1 \le p_2 < 1$, 
we can assume that 
$$\sup_{p \in [p_1, p_2]} C_p < +\infty.$$
\end{Rem}

Now we show Theorem \ref{thm:positive-parameterlimit}. 
\eqref{eq:positive-parameterlimit-gen} is equivalent with 
\begin{equation}\label{eq:positive-parameterlimit}
\lim_{p \to +0} \frac{S_{p,n}^{\frac{1}{p}} - S_{1,n}}{n^{\frac{1}{p}}-n} = \prod_{j=1}^{n} Y_j^{1/n}, \ \ \textup{ a.s. and in $L^1$}.  
\end{equation}

We first show the assertion for the case that $p=1/k$ for some positive integer $k$. 
Throughout the proof, $\ell$ is a positive number satisfying \eqref{ass-fine}. 

\begin{proof}[Proof of Theorem \ref{thm:positive-parameterlimit} for the case that $p=1/k$]
The almost sure convergence is obvious, so the rest of the proof is devoted to the $L^1$ convergence of \eqref{eq:positive-parameterlimit}. 

Let 
\[ A_k := \left\{ (a_1, \cdots, a_n) : a_1+\cdots+a_n = k, a_j \le k-1 \textup{ for every } j \right\}, \]
\[ B_{k,1} := \left\{ (a_1, \cdots, a_n) \in A_k : |a_j - k/n| \ge k/2n  \ \textup{ for some } j \right\}, \]
and $B_{k,2} := A_k \setminus B_{k,1}$.

By the large deviation result for the sum $ \frac{1}{k} \sum_{j=1}^{k} \left(\mathbf{1}_{\{i\}}(Z_j) - \frac{1}{n} \right)$,  
where we let $P(Z_1 = i) = 1/n, 1 \le i \le n,$ 
we see that for some $0 < \beta < 1$, 
\[ \frac{1}{n^k - n} \sum_{(a_1, \cdots, a_n) \in B_{k,1}} \frac{k!}{a_1! \cdots a_n !} \]
\begin{equation}\label{LDP-sum} 
= \frac{1}{n^{k-1} - 1} \sum_{a \le k-1; |a - k/n| \ge k/2n} \binom{k}{a} (n-1)^{k-a} = O(\beta^k). 
\end{equation}

By \eqref{ass-fine}, we see that 
$$E\left[|Y_j|^{a_j/k}\right] = O(k^\ell), \ a_j \le k-1.$$ 

By this and \eqref{LDP-sum}, we see that 
\[ \frac{1}{n^k - n} \sum_{(a_j)_j \in B_{k,1}} \frac{k!}{a_1! \cdots a_n !} E\left[ \prod_{j=1}^{n} |Y_j|^{a_{j}/k} \right] = O\left(k^{\ell n} \beta^k\right). \]

Let $\mathcal{A}$ be an arbitrarily event and consider 
\[ \sum_{(a_1, \cdots, a_n) \in B_{k,2}} \frac{k!}{a_1! \cdots a_n !} E\left[ \prod_{j=1}^{n} |Y_j|^{a_{j}/k}, \ \mathcal{A}\right].  \]

Recall that $n \ge 2$.
If $|a_i - k/n| < k/2n$, then, $a_i/k \le 3/4$.  
Hence, by the H\"older inequality and the inequality that $7a_j /6k \le 7/8$ if $|a_j - k/n| < k/2n$, 
we see that if $(a_1, \cdots, a_n) \in B_{k,2}$, 
\begin{align*} 
E\left[ \prod_{j=1}^{n} \left|Y_j\right|^{a_{j}/k}, \mathcal{A}\right] &\le E\left[ \prod_{j=1}^{n} |Y_j|^{\frac{7a_{j}}{6k}} \right]^{\frac{6}{7}} P(\mathcal{A})^{\frac{1}{7}} = P(\mathcal{A})^{\frac{1}{7}} \prod_{j=1}^n  E\left[|Y_1|^{\frac{7a_j}{6k}}\right]^{\frac{6}{7}} = O(1)  P(\mathcal{A})^{\frac{1}{7}}. 
\end{align*}

Thus we see that 
\[ E\left[\frac{1}{n^k - n} \sum_{(a_j)_j \in B_{k,2}} \frac{k!}{a_1! \cdots a_n !} \prod_{j=1}^{n} |Y_j|^{a_{j}/k},  \ \mathcal{A}\right] = O\left(k^{\ell n} \beta^k\right) + O(1) P(\mathcal{A})^{\frac{1}{7}}. \]

Since $k^{\ell n} \beta^k \to 0, \ k \to \infty$, 
we see that the sequence 
\[ \left(\frac{1}{n^k - n} \sum_{(a_j)_j \in B_{k,2}} \frac{k!}{a_1! \cdots a_n !} \prod_{j=1}^{n} Y_j^{a_{j}/k} \right)_k \]
is uniformly integrable {\it with respect to $k$. }
\end{proof}

Now we proceed to the general case. 

\begin{proof}[Proof of Theorem \ref{thm:positive-parameterlimit} for the case that $p \ne 1/k$]
The almost sure convergence is obvious, so the rest of the proof is devoted to the $L^1$ convergence of \eqref{eq:positive-parameterlimit}. 

Since we deal with the limit $p \to +0$, 
we can assume that $p < 1/100$. 
Let $k = k(p)$ be the unique integer such that $1/(k+3) < p \le 1/(k+2)$. 

We first see that 
\begin{equation}\label{eq:decomposition}
S_{p,n}^{\frac{1}{p}} - S_{pk,n} S_{1-pk,n} = \left( S_{p,n}^{k} - S_{pk,n} \right)  S_{p,n}^{\frac{1}{p}- k} + S_{pk, n}  \left( S_{p,n}^{\frac{1}{p}- k} -  S_{1-pk, n} \right). 
\end{equation}

We consider the first term of \eqref{eq:decomposition}. 
As in the proof of the case that $p=1/k$ above, 
we obtain that 
\[ \left( S_{p,n}^{k} -S_{pk, n} \right)  S_{p,n}^{\frac{1}{p}- k} = \sum_{(a_j)_j \in A_k} \frac{k!}{a_1! \cdots a_n !} \prod_{j=1}^{n} \left(Y_j^p\right)^{a_{j}} S_{p,n}^{\frac{1}{p}- k}.  \]

First we assume that $(a_j)_j \in B_{k,1}$. 
By the H\"older inequality and $2 \le 1/p - k < 3$, 
we see that 
\[ S_{p,n}^{1- kp} \le n^3 \sum_{j=1}^{n} |Y_j|^{1-kp}. \]
Hence, by the independence of $Y_1, \cdots, Y_n$, 
\[ E\left[ \prod_{j=1}^{n} \left(Y_j^p\right)^{a_{j}} S_{p,n}^{\frac{1}{p}- k}\right] \le n^3 \sum_{j=1}^{n} E\left[ |Y_j|^{1-kp + pa_j} \right] \prod_{m \ne j} E\left[ |Y_m|^{pa_{m}}\right]. \]
Since $2 \le 1/p - k < 3$ and $(a_j)_j \in A_k$, 
it holds that 
$1-kp + pa_j \le 1-p \le \frac{k+2}{k+3}$ and $pa_{m} \le \frac{k+2}{k+3}$ for each $m$.  
By this and \eqref{ass-fine}, 
we see that 
\[ E\left[ \prod_{j=1}^{n} \left(Y_j^p\right)^{a_{j}}S_{p,n}^{\frac{1}{p}- k}\right] = O\left( (k+3)^{\ell n} \right).\]
By this, $k < 1/p$ and \eqref{LDP-sum}, we see that 
\[ \frac{1}{n^{\frac{1}{p}} - n} \sum_{(a_j)_j \in B_{k,1}} \frac{k!}{a_1! \cdots a_n !} E\left[ \prod_{j=1}^{n} \left(|Y_j|^p\right)^{a_{j}} \left( \sum_{j=1}^{n} |Y_j|^{p}\right)^{\frac{1}{p}- k} \right] = O\left(k^{\ell n} \beta^k\right). \]

Second we assume that $(a_j)_j \in B_{k,2}$. 
Let $\mathcal{A}$ be an arbitrary event. 
Then, 
by the H\"older inequality and the assumption that $n \ge 2$, 
\[ E\left[ \left(\prod_{j=1}^{n} \left|Y_j\right|^{pa_{j}}\right) \left( \sum_{j=1}^{n} |Y_j|^{p}\right)^{\frac{1}{p}- k}, \ \mathcal{A}\right]  = O(1)  E\left[ \left(\sum_{j=1}^{n} |Y_j|^p   \right)^{7(\frac{1}{p} - k)}, \ \mathcal{A} \right]^{\frac{1}{7}}.\]

Since $p < 1/100$ and $1/p - k \le 2$, 
we see that 
\[ \lim_{P(\mathcal A) \to +0}  E\left[ \left(\sum_{j=1}^{n} |Y_j|^p   \right)^{7(\frac{1}{p} - k)}, \ \mathcal{A} \right] = 0. \]

Hence, 
the sequence $\left\{ \frac{1}{n^{\frac{1}{p}}} \left( S_{p,n}^{k} - S_{pk,n} \right)  S_{p,n}^{\frac{1}{p}- k}  \right\}_p$ 
is uniformly integrable  with respect to $p \in (0, 1/100)$.

We now consider the second term of \eqref{eq:decomposition}. 
By induction in $n$, we will show that 
\begin{equation}\label{eq:induction}
\lim_{p \to +0} \frac{1}{n^{\frac{1}{p}}} S_{pk,n}  \left( S_{p,n}^{\frac{1}{p}- k} -  S_{1-pk,n} \right) = 0,  \ \textup{ a.s. and in $L^1$. }
\end{equation} 

The almost sure convergence holds, because $2 \le 1/p - k < 3$, and, $pk \to 1$ as $p \to +0$.  
Now we show the $L^1$ convergence.   
Let $n=2$. 
Then, $\max\{p, 1-(k+1)p\} = 1-(k+1)p$. 
By this and Lemma \ref{lem:positive-fundamental}, it holds that 
\[ \left| S_{p,n}^{\frac{1}{p}- k} -  S_{1-pk, n} \right| \le C_{\frac{p}{1-kp}} (1+|Y_1|^{1-(k+1)p})(1+|Y_2|^{1-(k+1)p}). \]
Hence, 
\[ \left| S_{pk, n} \right| \left| S_{p,n}^{\frac{1}{p}- k} - S_{1-pk, n} \right| \]
\[ \le C_{\frac{p}{1-kp}} \left( (|Y_1|^{pk}+|Y_1|^{1-p}) (1+|Y_2|^{1-(k+1)p}) + (1+|Y_1|^{1-(k+1)p})(|Y_2|^{pk}+ |Y_2|^{1-p}) \right). \]
Since $1/3 \le \frac{p}{1-kp} \le 1/2$, 
by Remark \ref{rem:unif-bdd}, 
$\sup_{p} C_{\frac{p}{1-kp}} < +\infty$.
By this and \eqref{ass-fine},
we obtain \eqref{eq:induction} for $n=2$. 

Assume that \eqref{eq:induction} holds for $n \le m$. 
We will show \eqref{eq:induction} for $n=m+1$. 
Then, 
\begin{align*} 
\left| S_{p,m+1}^{\frac{1}{p}- k} - S_{1-pk, m+1} \right|  &\le \left| S_{p, m+1}^{\frac{1}{p}- k} -  S_{p, m}^{\frac{1}{p} - k} - Y_{m+1}^{1- kp} \right| +  \left| S_{p,m}^{\frac{1}{p}- k} -  S_{1-pk, m}  \right|. 
\end{align*}

By the inductive assumption, 
it holds that 
\[  \lim_{p \to +0} \frac{1}{(m+1)^{\frac{1}{p}}}  \left|S_{pk, m}\right|  \left| S_{p, m}^{\frac{1}{p}- k} -  S_{1-pk, m} \right|  = 0, \ \textup{ in $L^1$. } \] 

By \eqref{ass-fine} and $2p \le 1-kp \le 3p$, 
we also see that 
\[  \lim_{p \to +0} \frac{1}{(m+1)^{\frac{1}{p}}}  \left| Y_{m+1}^{pk} \right|  \left| S_{p,m}^{\frac{1}{p}- k} - S_{1-pk, m}\right|  = 0, \ \textup{ in $L^1$.} \]  

Hence, it suffices to show that 
\[  \lim_{p \to +0} \frac{1}{(m+1)^{\frac{1}{p}}}  
\left| S_{pk, m+1} \right|  \left| S_{p, m+1}^{\frac{1}{p}- k} - S_{p,m}^{\frac{1}{p} - k} - Y_{m+1}^{1- kp} \right| = 0, \ \textup{  in $L^1$. } \]

It holds that 
\[ \left| S_{p, m+1}^{\frac{1}{p}- k} -  S_{p,m}^{\frac{1}{p} - k} - Y_{m+1}^{1- kp} \right|  \le C_{\frac{p}{1-kp}} \left(1+ \left| S_{p,m} \right|^{1/p-(k+1)}\right) \left(1+|Y_{m+1}|^{1-(k+1)p} \right). \]

Hence, 
\[  \left| S_{pk, m+1} \right| \cdot \left| S_{p, m+1}^{\frac{1}{p}- k} -  S_{p,m}^{\frac{1}{p} - k} - Y_{m+1}^{1- kp} \right|\]
\[ \le C_{\frac{p}{1-kp}} \left(1+ \left|S_{p,m}\right|^{\frac{1}{p}-(k+1)}\right) \left( \left| S_{pk,m} \right| \left(1+|Y_{m+1}|^{1-(k+1)p} \right) + |Y_{m+1}|^{kp} +|Y_{m+1}|^{1-p}  \right). \] 

Since $1 \le 1/p-(k+1) \le 2$, $kp  \le 1-p$ and \eqref{ass-fine}, 
we see that by the H\"older inequality, 
\[ E\left[ \left| S_{pk, m} \right| \left|S_{p,m} \right|^{\frac{1}{p}-(k+1)} \right] \le m^2  E\left[ \left( \sum_{j=1}^{m} |Y_j|^{pk} \right) \left(\sum_{j=1}^{m} |Y_j|^{1- (k+1)p} \right) \right] \]
\[\le m^3 E\left[|Y_1|^{1-p}\right] + m^4 E\left[|Y_1|^{pk} \right]  E\left[|Y_1|^{1 - (k+1)p} \right] = m^4 O(p^{-\eta}), \]
where we should recall that $\eta$ comes from \eqref{ass-fine}. 
Thus we have \eqref{eq:induction} for $n=m+1$.

Finally, by \eqref{ass-fine}, 
we see that 
\[ \lim_{p \to +0} \frac{1}{n^{\frac{1}{p}}} \left( S_{1,n}  - S_{pk,n} S_{1-pk,n} \right) = 0, \  \textup{ a.s. and in $L^1$, } \] 
in the same manner as above.

Thus we also see that 
\[ \lim_{p \to +0} \frac{1}{n^{\frac{1}{p}}}  \left( S_{p,n}^{k} -S_{pk,n}\right)  S_{p,n}^{\frac{1}{p}- k}  = \prod_{j=1}^{n} Y_j^{1/n}, \ \textup{ a.s. } \] 
By recalling the uniform integrability, 
the above convergence holds in $L^1$. 
Thus we have the $L^1$ convergence of \eqref{eq:positive-parameterlimit}. 
\end{proof}

\section{Limit theorems for sums of products of random variables}\label{sec:pdsum}

The truncated power means in Theorem \ref{thm:truncated-integrability} are in $L^1$, but it is easy to see that they are not in $L^2$ for $p=1/k$. 
In this section, we execute further truncations in order to obtain the integrability of order $2$.

\begin{Def}[sums of products]
\[ R^{(\alpha)}_{m,n} :=  \frac{1}{\binom{n}{m}} \sum_{1 \le \ell_1 < \cdots < \ell_m \le n} \left(\prod_{j=1}^{m} (X_{\ell_j} + \alpha)^{\frac{1}{m}} - \alpha\right), \ \ n \ge m.  \] 
We remark that $R^{(\alpha)}_{m, m} = G^{(\alpha)}_m$.  
\end{Def}

We first recall the following: 
\begin{Thm}[Variance asymptotics for geometric mean {\cite[Theorem 2.1]{Akaoka2021-1}}]\label{var-conv}
If $X_1 \in L^{0+}$, then, 
\[ \lim_{n \to \infty}  n \textup{Var}\left( G^{(\alpha)}_n \right) = \exp(2E[\log |X_1 + \alpha|])   \textup{Var}\left(\log (X_1 + \alpha) \right). \]
\end{Thm}

We can show the above by using Theorem \ref{thm:gen-var-lim} in Section \ref{sec:asym-var}.

Now we consider  asymptotic behaviors for sums of products $\left(R^{(\alpha)}_{m, n}\right)_{n \ge m}$. 
As the following shows, their behaviors are similar to $G_m^{(\alpha)}$. 

\begin{Thm}\label{PRS}
Let $m \ge 2$ and $X_1 \in L^{\frac{2m}{2m+1}}$. Then, \\
(i) The following convergence holds almost surely and  in $L^{m-1}$:
\[  \lim_{n \to \infty} R^{(\alpha)}_{m, n} = E\left[ G_{m}^{(\alpha)}\right].\]
(ii) For $n \ge m$, 
\[ m \textup{Var} \left(G^{(\alpha)}_m \right) \frac{ \left| E\left[ (X_1 + \alpha)^{\frac{1}{m}} \right] \right|^{2m}}{E\left[|X_{1} + \alpha|^{\frac{2}{m}} \right]^m}  \le n \textup{Var}\left( R^{(\alpha)}_{m, n}  \right)  
\le m \textup{Var} \left(G^{(\alpha)}_m\right). \]
\end{Thm}

\begin{proof}
(i) The a.s. convergence follows from an application of \cite[Proposition 4.1 (i)]{Akaoka2021-1} to the case that $f(x) = (x+\alpha)^{\frac{1}{m}}$. 
Let 
\[ \mathcal{D} = \mathcal{D}_{n,m} := \{(\ell_1, \cdots, \ell_m) \in \{1, \dots, n\}^m : \textup{distinctive}\}\]
and 
\[ \mathcal{N} = \mathcal{N}_{n,m} := \{(\ell_1, \cdots, \ell_m) \in \{1, \dots, n\}^m : \textup{non-distinctive}\}.\]

We also see  that 
\[  \frac{1}{\binom{n}{m}} \sum_{1 \le \ell_1 < \cdots < \ell_m \le n} \prod_{j=1}^{m} Y_{\ell_j}^{\frac{1}{m}}  =  \frac{1}{m!\binom{n}{m}} \sum_{(\ell_1, \cdots, \ell_m) \in \mathcal{D}} \prod_{j=1}^{m} Y_{\ell_j}^{\frac{1}{m}} .  \]
We see that 
\[ \left(\sum_{i=1}^{n} Y_i^{\frac{1}{m}}\right)^m = \sum_{(\ell_1, \cdots, \ell_m) \in \mathcal{N}}  \prod_{j=1}^{m} Y_{\ell_j}^{\frac{1}{m}}  + \sum_{(\ell_1, \cdots, \ell_m) \in \mathcal{D}} \prod_{j=1}^{m} Y_{\ell_j}^{\frac{1}{m}} .  \]

It is easy to see that 
\[ |\mathcal{D}| = n(n-1) \cdots (n -(m-1)) = m! \binom{n}{m}. \]

We see that for every $a > 1$, 
\[ \lim_{n \to \infty} \frac{1}{n^{a}} \sum_{j=1}^{n} Y_j = 0, \  \textup{ a.s.} \]

Hence, 
\[ \frac{1}{n^m} \left|\sum_{(\ell_1, \cdots, \ell_m) \in \mathcal{N}}  \prod_{j=1}^{m} Y_{\ell_j}^{\frac{1}{m}} \right| \le \left( \frac{1}{n^{\frac{2m^2}{2m+1}}} \sum_{(\ell_1, \cdots, \ell_m) \in \mathcal{N}}  \prod_{j=1}^{m} |Y_{\ell_j}|^{\frac{2}{2m+1}} \right)^{1+\frac{1}{2m}}.  \]

By the inequality of the arithmetic mean and the geometric mean, 
\[ \frac{1}{n^{\frac{2m^2}{2m+1}}} \sum_{(\ell_1, \cdots, \ell_m) \in \mathcal{N}}  \prod_{j=1}^{m} |Y_{\ell_j}|^{\frac{2}{2m+1}} \le \frac{1}{ m n^{\frac{2m^2}{2m+1}}} \sum_{(\ell_1, \cdots, \ell_m) \in \mathcal{N}}  \sum_{j=1}^{m} |Y_{\ell_j}|^{\frac{2m}{2m+1}}.  \]
Since $|\mathcal{N}| = n^m - \binom{n}{m},$ 
the total number of each $i \in \{1, \cdots, n\}$ appearing in the set $\mathcal{N}$ is $(n^m - \binom{n}{m})/n$, and hence 
we see that 
\[ \sum_{(\ell_1, \cdots, \ell_m) \in \mathcal{N}} \sum_{j=1}^{m} |Y_{\ell_j}|^{\frac{2m}{2m+1}}  = \frac{n^m - \binom{n}{m}}{n} \sum_{j=1}^{m} |Y_{\ell_j}|^{\frac{2m}{2m+1}}. \]

Since $|\mathcal{N}| = O(n^{m-1})$,
we see that by the strong law of large numbers, 
\[ \frac{1}{ m n^{2m^2 / (2m+1)}} \frac{n^m - \binom{n}{m}}{n} \sum_{j=1}^{m} |Y_{\ell_j}|^{\frac{2m}{2m+1}} = O(m^{-1}) \frac{1}{n^{3/2}} \sum_{j=1}^{n} |Y_{\ell_j}|^{\frac{2m}{2m+1}} \to 0, \ \ n  \to \infty, \textup{ a.s.} \]
Thus we see that 
\[ \lim_{n \to \infty} \frac{1}{n^m} \sum_{(\ell_1, \cdots, \ell_m) \in \mathcal{N}}  \prod_{j=1}^{m} Y_{\ell_j}^{\frac{1}{m}}  = 0,  \textup{ a.s.} \]

The $L^{m-1}$-convergence follows from the above a.s. convergence and the following uniform integrability. 
By Jensen's inequality, we see that for every $n$, 
\begin{align*} 
E\left[ \left|   \frac{1}{\binom{k}{m}} \sum_{1 \le \ell_1 < \cdots < \ell_m \le n} \prod_{j=1}^{m} Y_{\ell_j}^{\frac{1}{m}}  \right|^{m-1/2} \right] 
&\le  E\left[  \frac{1}{\binom{k}{m}} \sum_{1 \le \ell_1 < \cdots < \ell_m \le n} \prod_{j=1}^{m} |Y_{\ell_j}|^{\frac{2m-1}{2m}} \right] \\
&= E\left[|Y_{1}|^{\frac{2m-1}{2m}} \right]^m.  
\end{align*}

(ii) The value of $E\left[\prod_{j=1}^{m} Y_{\ell_j}^{\frac{1}{m}} \overline{\prod_{j=1}^{m} Y_{\ell_j^{\prime}}^{\frac{1}{m}} }\right]$ depends on 
the overlap between $\ell_1 < \cdots < \ell_m$ and $\ell_1^{\prime} < \cdots < \ell_m^{\prime}$. 
Let the overlap number be the number of $\ell_i$ such that there exists $\ell_j^{\prime}$ such that $\ell_i = \ell_j^{\prime}$. 
We see that if the overlap number is $k$, then, 
\[ E\left[\prod_{j=1}^{m} Y_{\ell_j}^{\frac{1}{m}} \overline{\prod_{j=1}^{m} Y_{\ell_j^{\prime}}^{\frac{1}{m}} }\right] =  E\left[|Y_{1}|^{\frac{2}{m}}\right]^k  \left|E\left[Y_1^{\frac{1}{m}}\right]\right|^{2(m-k)}. \]

The overlap number is more likely to be small as $n \to \infty$.
By a combinatorial argument, we see that 
\[ \frac{1}{\binom{n}{m}^2} \sum_{1 \le \ell_1 < \cdots < \ell_m \le n; 1 \le \ell_1^{\prime} < \cdots < \ell_m^{\prime} \le n} E\left[\prod_{j=1}^{m} Y_{\ell_j}^{\frac{1}{m}} \overline{\prod_{j=1}^{m} Y_{\ell_j^{\prime}}^{\frac{1}{m}} }\right]\] 
\[ = \frac{1}{\binom{n}{m}} \sum_{i=0}^{m} \binom{m}{i} \binom{n-m}{m-i}  E\left[|Y_{1}|^{\frac{2}{m}} \right]^i  \left|E\left[ Y_1^{\frac{1}{m}}\right]\right|^{2(m-i)}. \]

By this and the Chu-Vandermonde identity \cite{Askey1975}, 
we see that 
\[ \textup{Var}\left( \frac{1}{\binom{n}{m}} \sum_{1 \le \ell_1 < \cdots < \ell_m \le n} \prod_{j=1}^{m} Y_{\ell_j}^{\frac{1}{m}}   \right) \]
\[= \frac{1}{\binom{n}{m}^2} \sum_{1 \le \ell_1 < \cdots < \ell_m \le n; 1 \le \ell_1^{\prime} < \cdots < \ell_m^{\prime} \le n} E\left[\prod_{j=1}^{m} Y_{\ell_j}^{\frac{1}{m}} \overline{\prod_{j=1}^{m} Y_{\ell_j^{\prime}}^{\frac{1}{m}} }\right] -  \left|E\left[ Y_1^{\frac{1}{m}}\right]\right|^{2m}. \]

Now we give an estimation for the right hand side of the above display.   
Let $a_m := E\left[ \left| Y_{1} \right|^{\frac{2}{m}} \right]^m$ and $b_m := \left|E\left[ Y_1^{\frac{1}{m}} \right]\right|^{2m}.$ 
Then,  $a_m \ge  b_m$ for each $m \ge 2$, $(a_m)_m$ is decreasing, and 
\[ a_m - b_m = \textup{Var}(G_m^{(\alpha)}).  \]

Therefore, by the mean-value theorem, we see that 
\[ \frac{i}{m} (a_m)^{\frac{i}{m} - 1} (a_m - b_m)  \le (a_m)^{\frac{i}{m}} - (b_m)^{\frac{i}{m}} \le \frac{i}{m} (b_m)^{\frac{i}{m} - 1}  (a_m - b_m), \ 0 \le i \le m.  \]

By this and the Chu-Vandermonde identity again,  
we see that 
\[ \left| \frac{1}{\binom{n}{m}} \sum_{i=0}^{m} \binom{m}{i} \binom{k-m}{m-i}  \left(E\left[|Y_{1}|^{\frac{2}{m}} \right]^i - \left|E\left[ Y_1^{\frac{1}{m}} \right]\right|^{2i}\right) \left|E\left[ Y_1^{\frac{1}{m}} \right]\right|^{2(m-i)} \right| \]
\[ \le \frac{a_m - b_m}{\binom{n}{m}} \sum_{i=1}^{m} \binom{m-1}{i-1} \binom{k-m}{m-i}  = \frac{m(a_m - b_m)}{n}. \]

On the other hand, 
\[ \frac{1}{\binom{n}{m}} \sum_{i=0}^{m} \binom{m}{i} \binom{n-m}{m-i}  \left(E\left[| Y_{1}|^{\frac{2}{m}} \right]^i -  \left|E\left[ Y_1^{\frac{2}{m}}\right]\right|^{i}\right)  \left|E\left[ Y_1^{\frac{2}{m}}\right]\right|^{m-i} \ge \frac{m b_m (a_m - b_m)}{k a_m}. \]

Now we have (ii). 
\end{proof}

\section{Cauchy distribution}\label{sec:Cauchy}

Now we restrict our interest to the random variables following the Cauchy distribution. 
The results in the above sections are stated under weak integrability conditions for $X_1$ and are applicable to the Cauchy distribution. 
Before giving the results and their proofs, we state some background, in particular, statistical properties of a quasi-arithmetic mean $M^f_n$ of i.i.d.\ Cauchy random variables as an estimator of the joint of the location and scale parameters of the Cauchy distribution.  
We assume that the generator $f$ is the power mean with parameter $p \in [-1,1)$. 
Recall \eqref{eq:def-generator}. 

\subsection{Background}

Many estimation methods have been considered, including order statistics and maximum likelihood estimation. 
These results try to balance the computational complexities with the consistency, efficiency, and robustness of estimators. 
For reviews, we refer the reader to \cite[Section 4]{Akaoka2021-1} and \cite[Chapter 16]{Johnson1994} for the results obtained before 1994. 

We now give statistical properties of a quasi-arithmetic mean $M^f_n$ as an estimator of the location-scale Cauchy family.    
McCullagh's parametrization for the Cauchy distribution (\cite{McCullagh1996}) is crucial. 
Specifically, we regard the joint of the location parameter $\mu$ and the scale parameter $\sigma$  as {\it a single complex number} $\gamma = \mu + i\sigma$. 
This idea also appeared in Letac \cite{Letac1978}. 
The estimators $\left(M^f_n\right)_n$  work well for estimations of the {\it joint} of the location and scale parameters of the Cauchy distribution, and not for estimations of location or scale separately. 

$M$-estimators are a general framework containing all of the maximum likelihood estimator(MLE), the mean, and the median as special cases. 
In general, several methods for {\it simultaneous} estimation of location and scale have been investigated, however, they are much more complicated than estimating either location or scale separately. 
See \cite{Huber2009, Maronna2019} for robust statistics. 
The random variable $M_n^f$ can be regarded as  an $M$-estimator of $\gamma$ defined as a solution of 
$\sum_{i=1}^{n} \Psi(X_i, \gamma) = 0$,
where the score function $\Psi$ is defined by $\Psi(x, \gamma) := f(x) - f(\gamma)$. 
We see that $ \sum_{i=1}^{n} \Psi(X_i, M^f_n) = 0$. 
It is Fisher-consistent, that is, $E_{\gamma} [\Psi(X_1, \gamma)] = 0$, 
and the influence function associated with the $M^f_n$ is bounded. 
Thus, $M^f_n$ is a robust estimator of the joint of the location and scale parameters. 
For some one-dimensional distributions {\it other than} the Cauchy distribution,  
$M^f_n$ is robust for $p \ne 0$ (\cite{Mukhopadhyay2021}) and for $p=0$ (\cite{Akaoka2021-3}).  
On the other hand, it should be noted that the Cauchy distribution is an outlier model, so samples with large values are {\it not inconsistent}. 

As mentioned in \cite[Section 6.4]{Huber2009} and \cite[Section 2.7]{Maronna2019}, 
one way to make the simultaneous estimation of location and scale is to use the MLE. 
The MLE of the joint of the location and scale parameters is also robust, however, as mentioned in \cite[Theorem 17]{Okamura2021}, 
we cannot expect explicit algebraic formulas for the sample size $n \ge 5$ and numerical computations such as the Newton-Raphson method may not work well. 
However, on the other hand, the MLE has a strong connection with the power mean of $p=-1$. 
For $\alpha \in \mathbb{H}$, let $\phi_{\alpha}(x) := (x + \alpha)^{-1}$. 
Fix the sample size $n$ and let $Q(\theta) = Q_n (\theta) := M^{\phi_{-\overline{\theta}}}_n$. 
Then, by \cite[Theorem 15]{Okamura2021}, 
the MLE is characterized by a unique fixed point of $Q$.  
Furthermore, if we let $Y_m := Q(Y_{m-1})$, $m \ge 1$, $Y_0 = \beta \in \mathbb H$,  then, 
$\displaystyle\lim_{m \to \infty} Y_m$ is equal to the MLE for every starting point $\beta$ almost surely.  
$Q$ can be seen as a random holomorphic function, and in this sense, we can regard $(\mathbb{H}, Q)$ as a random dynamical system.   
Furthermore, if $\alpha$ is close to $-\overline{\gamma}$, then, the asymptotic behavior of the variances of $\left(M^{\phi_{\alpha}}_n\right)_n$ can be {\it arbitrarily close to} the Cramer-Rao lower bound. 
See \cite[Remark 4.5]{Akaoka2021-1} for details.   

By  \eqref{eq:scale-equivariant} above,  for $\alpha = 0$,  $(M^f_n)_n$ are scale equivariant estimators. 
It is easy to see that $(M^f_n)_n$ is a $\sqrt{n}$-consistent, unbiased, strongly consistent estimator of the Cauchy distribution for $-1 \le p < 0$. 
We obtain the asymptotic behavior for the variances of $(M^f_n)_n$.  
The integrability results for $(M^f_n)_n$ are applicable to the Cauchy distribution. 
In particular, $M^f_n$ has a finite variance for $n \ge 3$. 
We remark that the variance of the MLE diverges for $n=3$ (\cite{McCullagh1996}). 
However, for $0 < p < 1$, we need to truncate some terms from $M^f_n$ in order to assure integrabilities. 
We can also construct confidence discs of $\gamma$ with explicit formulas for the center and the radius of the disc, and, explicitly compute the inaccuracy rate of the large and moderate deviations. 

Our approach leads to the geometric understanding of the estimator. 
$(M^f_n)_n$ is a quasi-likelihood estimator and $\Psi(x, \gamma) = f(x) - f(\gamma)$ is a non-linear differentiable unbiased estimating function. 
We can consider the information geometry based on the estimating function as in \cite{Henmi2011}. 
More specifically, a standardized estimation function defines a Riemannian metric on the space of parameters called the Godambe information matrix. 
It plays the same role as the Fisher information matrix, so that the structure of a statistical manifold is given in the space of parameters. 
Information-geometric properties of the statistical manifold with the Fisher information matrix of the location-scale Cauchy family have been recently investigated by \cite{Nielsen2020}. 
\cite{Itoh2015} introduced a barycenter associated with the Busemann function on Hadamard manifolds. 
The Busemann function is defined on the statistical manifold of the location-scale Cauchy family in a closed form. 
By \cite[Example 2]{Itoh2015}, the MLE coincides with the barycenter of the empirical probability measure of a sample. 
The considerations of statistical (complex) manifolds associated with $f$ and geometric characterizations of $M^f_n$ are interesting open problems.

\subsection{Results}

Let $p_{\gamma}(x)$ be the density function of $X_1$, specifically, 
\[ p_{\gamma}(x) = \frac{\sigma}{\pi} \frac{1}{(x-\mu)^2 + \sigma^2} = \frac{1}{2\pi i}\left( \frac{1}{x-\gamma} - \frac{1}{x-\overline{\gamma}}\right), \ \ x \in \mathbb{R},  \]
where $\mu$ and $\sigma$ are the real and imaginary parts of $\gamma$ respectively. 
Throughout this section, we let $X_1, \cdots, X_n, \dots, $ be i.i.d.\ Cauchy random variables with a complex parameter $\gamma \in \mathbb H$. 

For every function $H$ in a class of holomorphic functions on the upper-half plane and for every $\gamma \in \mathbb H$, 
it holds that 
\begin{equation}\label{eq:Cauchy-unbiased} 
E[H(X_1)] = H(\gamma) 
\end{equation} 
if $X_1$ follows the Cauchy distribution with parameter $\gamma$, see \cite{McCullagh1996} and \cite[Section 4]{Akaoka2021-1}. 

Assume that either (i) $-1 < p < 0$ and $\alpha \in \mathbb H$
or (ii) $-1/2 < p < 0$ and $\alpha \in \mathbb R$  
holds. 
We recall that $f(x) = (x+\alpha)^p$. 
Then, by the residue theorem, 
we see that \eqref{eq:Cauchy-unbiased} holds for $H = f$ and $H = f^2$. 
See \cite[Proposition 4.1 (i)]{Akaoka2021-1} for more details. 
Hence, $f(X_1) - f(\gamma)$ is {\it proper} as a complex-valued random variable, specifically, $f(X_1) \in L^2$, and $E[f(X_1) - f(\gamma)] = E\left[ (f(X_1) - f(\gamma))^2 \right] = 0$.   
We also see that 
\begin{equation}\label{eq:cov-nonsingular}
\textup{Cov}\left(f(X_1)\right) = \frac{\textup{Var}(f(X_1))}{2} I_2, 
\end{equation}
and 
\begin{equation}\label{eq:cov-easy} 
J(f^{-1}) \textup{Cov}(f(X_1)) J(f^{-1})^{\prime} = \frac{\textup{Var}(f(X_1))}{2|f^{\prime}(\gamma)|^2} I_2, 
\end{equation}
where $I_2$ is the unit matrix of degree $2$.

Let $\mathcal{P}(\mathbb R)$ be the set of Borel probability measures on $\mathbb{R}$.  
For a map $T : \mathcal{P}(\mathbb R) \to \mathbb H$ and $P \in \mathcal{P}(\mathbb{R})$, 
we define the influence function $\textup{IF}(x;T,P)$ by
\[ \textup{IF}(x;T,P) := \left.\frac{d}{d\varepsilon} \right|_{\varepsilon = 0} T\left((1-\varepsilon) P + \varepsilon \delta_x\right), \ x \in \mathbb{R}. \]

We now let $\displaystyle T(P) := f^{-1}\left(\int_{\mathbb R} f(x) P(dx)\right)$. 
Let $f(x) = (x+\alpha)^p, -1 \le p < 0, \alpha \in \mathbb H$. 
Then, 
\[ \textup{IF}(x;T,P_{\gamma}) = \frac{f(x) - f(\gamma)}{f^{\prime}(\gamma)} = \frac{(x+\alpha)^p - (\gamma + \alpha)^p}{p(\gamma + \alpha)^{p-1}}, \]
where $P_{\gamma}(dx) := p_{\gamma}(x) dx$. 
Since $(x+\alpha)^p \to 0, \ x \to \pm\infty$, 
$\textup{IF}(x;T,P_{\gamma})$ is bounded for $x$. 
Thus our estimators are robust for the joint of the location and scale parameters. 

Now we state our results. 

\begin{Thm}[$L^1$ integrability and unbiasedness for negative parameter case]\label{thm:negative-unbiased}
Let $n \ge 2$. 
Then the following assertions hold:\\
(i) Let $p \in (-1,0)$ and $\alpha \in \overline{\mathbb H}$. 
Then, $M^f_n \in L^1$ and $E\left[M^f_n\right] = \gamma$. \\ 
(ii) Let $p=-1$ and $\alpha \in \mathbb R$. 
Then, $M^f_n \notin L^1$.\\
(iii) Let $p=-1$ and $\alpha \in \mathbb H$. Then, 
$M^f_n \in L^1$ and $E\left[M^f_n\right] = \gamma$. 
\end{Thm}

The negative power means of i.i.d.\ Cauchy random variables are unbiased estimators if $n \ge 2$. 
However, the above assertion fails in the case that $n=1$.

\begin{Thm}[$L^2$ integrability for negative parameter case]\label{thm:negative-Cauchy-var}
The following assertions hold.\\
(i) Let $n=2$, $-1 \le p < 0$ and $\alpha \in \overline{\mathbb H}$. 
Then, $M^f_n \notin L^2$.\\
(ii) Let $n=3$, $p=-1$ and $\alpha \in \overline{\mathbb H}$. 
Then, $M^f_n \notin L^2$.\\
(iii) Let $n \ge 3$, $-1 < p < 0$ and $\alpha \in \overline{\mathbb H}$. 
Then, $M^f_n \in L^2$.\\
(iv) Let $n \ge 4$, $p = -1$, and $\alpha \in \mathbb H$.  
Then, $M^f_n \in L^2$.
\end{Thm}

For $n=3$, there is a difference between the case that $p > -1$ and the case that $p=-1$. 
It seems to be possible to extend those results under weak integrability assumptions for $X_1$, however, our proofs use the symmetry of the density function of the Cauchy distribution and it enables us to simplify the proofs. 

Let 
$$V(p) := \frac{\textup{Var}(f(X_1))}{\left|f^{\prime}(\gamma) \right|^2} = \frac{E\left[ |X_1 + \alpha|^{2p} \right] - |\gamma + \alpha|^{2p}}{p^2} |\gamma + \alpha|^{2(1-p)}.$$
The following are corollaries to Theorem \ref{thm:neg-var-lim} and Theorem \ref{thm:negative-unbiased}.  
\begin{Cor}\label{cor:asympvar-Cauchy}
(i) Assume that either (a) $-1 < p < 0$ and $\alpha \in \mathbb H$
or (b) $-1/2 < p < 0$ and $\alpha \in \mathbb R$   
holds. 
Then, 
\[ \lim_{n \to \infty} n \textup{Var}\left(M^f_n \right) = V(p). \]
(ii) If  $-1 < p \le -1/2$ and $\alpha \in \mathbb R$, 
Then, 
\[ \lim_{n \to \infty} n \textup{Var}\left(M^f_n \right) = +\infty. \]
\end{Cor}

It is natural to consider the changes in the values of the asymptotic variances as $p$ and $\alpha$ vary. 
Since $M_n^f$ is an unbiased estimator,  by the Cram\'er-Rao inequality, 
it holds that for every $n \ge 1$, 
\[ n \textup{Var}\left(M^f_n \right) \ge 4\sigma^2, \]
see \cite[(33)]{Akaoka2021-1}. 
By Corollary \ref{cor:asympvar-Cauchy}, 
the estimator $\left(M^f_n\right)_n$ is not asymptotically efficient. 
We investigate the value of $p$ which minimizes $V(p)$. 

\begin{Prop}\label{prop:asymp-var-vary}
(i) $$\lim_{p \to 1/2-0} V(p)  = +\infty.$$  \\
(ii) $$\lim_{p \to -0} V(p) = \exp(2E[\log |X_1 + \alpha|])   \textup{Var}\left(\log (X_1 + \alpha) \right).$$\\
(iii) If $\alpha \in \mathbb R$, then, 
$V$ is strictly convex and decreasing on  $(-1/2, 0)$.  \\
(iv) If $\alpha = \gamma = i$, then, 
$$V(p) \ge V(-1) = 4, -1 \le p < 0$$ 
and 
$$\lim_{p \to -0} V(p) =  \frac{2\pi^2}{3}.$$   
\end{Prop}

If $\alpha \in \mathbb R$, then, the asymptotic variance of the geometric mean takes smaller than that of the negative power means, as is announced in the introduction in \cite{Akaoka2021-3}. 

We deal with the large and moderate deviations of $\left(M_n^f \right)_n$.  
\begin{Thm}[inaccuracy rate]\label{thm:inaccu}
Assume that either (a) $-1 < p < 0$ and $\alpha \in \mathbb H$
or (b) $p = 0$ and $\alpha \in \overline{\mathbb H}$  
holds. 
Then, \\
(i) 
\[ \lim_{\varepsilon \to +0} \frac{1}{\varepsilon^2}  \left(\lim_{n \to \infty} \frac{1}{n} \log P\left( \left| M^f_n   - \gamma \right| > \varepsilon\right)\right)  = -\frac{\left| f^{\prime} \left( \gamma \right) \right|^2}{\textup{Var} (f(X_1))}. \]
(ii) 
Let $(c_n)_n$ be a sequence of positive numbers such that $\lim_{n \to \infty} c_n = 0$ and $\lim_{n \to \infty} n c_n^2 = +\infty$. 
Then, 
\[ \lim_{n \to \infty} \frac{1}{n c_n^2} \log P\left( \left| M^f_n   - \gamma \right| > c_n\right)  = -\frac{\left| f^{\prime} \left( \gamma \right) \right|^2}{\textup{Var} (f(X_1))}. \]
\end{Thm}

Assertion (i) is the Bahadur efficiency for the quasi-arithmetic means which is mentioned in \cite{Akaoka2021-2} without proofs. 
The limit, $\frac{\left| f^{\prime} \left( \gamma \right) \right|^2}{\textup{Var} (f(X_1))}$, is called an {\it inaccuracy rate}, has a very simple form due to the assumption that $f(X_1) - f(\gamma)$ is proper. 

Assume that $f(X_1) \in L^2$. 
Let
\[ V_n^f := \frac{1}{n} \sum_{j=1}^n |f(X_j)|^2 - \left| \frac{1}{n}\sum_{j=1}^n f(X_j) \right|^2 \to \textup{Var}(f(X_1)), \ n \to \infty, \ \textup{ a.s.} \]

For $a \in (0, 1/2)$, 
we let $R_a > 0$ be the constant such that $P(|Z| > R_a) = a$ for $Z \sim N(0, I_2)$. 
Let $B(x,r)$ be the open ball with center $x$ and radius $r$ with respect to the Euclidian norm. 

By using \eqref{eq:delta-CLT}, \eqref{eq:cov-easy} and Slutsky's lemma, 
we have that 
\begin{Prop}\label{prop:cd-extension}
If $f(X_1) \in L^2$, then, 
\[ \lim_{n \to \infty} P\left( \gamma \in B\left(M^f_n, \frac{\sqrt{V^f_n}}{\sqrt{2n} \left|f^{\prime}(M^f_n) \right|} R_a \right) \right) = 1-a. \]
\end{Prop}

This is a generalization of \cite[Theorem 3.5]{Akaoka2021-3}. 
The asymptotic radius is $$\frac{\sqrt{V^f_n}}{\sqrt{2n} \left|f^{\prime}(M^f_n) \right|} R_a \sim \frac{\sqrt{\textup{Var}(f(X_1))}}{\sqrt{2n} \left|f^{\prime}(\gamma) \right|} R_a = \sqrt{V(p)} \frac{R_a}{\sqrt{2n}}, \ \ n \to +\infty,$$
where $a_n \sim b_n$ means that $a_n / b_n \to 1, n \to \infty$ for sequences $(a_n)_n, (b_n)_n$. 
Therefore, if $\gamma$ is closed to $i$, then, the choice that $f(x) = f^{(i)}_{-1}(x) = 1/(x+i)$ is better than the choice that $f(x) = \log x$ in \cite{Akaoka2021-3}. 
Specifically, $V(-1) = 4$ if $f(x) = f^{(i)}_{-1}(x)$, but on the other hand $\lim_{p \to -0} V(p) = \pi^2 /2$ if $f(x) = \log x$. 
As in \cite{Akaoka2021-3}, we can also consider squares and strips as confidence regions. 

Finally, we deal with the case that $p > 0$. 

\begin{Cor}[truncated positive power means for Cauchy distribution]\label{cor:truncated-Cauchy}
(i) (unbiasedness) Let $n \ge 2$. 
Let $0 < p < 1$. 
Let $\alpha \in \overline{\mathbb H}$.
Then, \eqref{eq:trancated-L1} holds and furthermore, 
\begin{equation*}
E\left[ \frac{n^{1/p} M^{f^{(\alpha)}_p}_n - n M^{f^{(\alpha)}_1}_n}{n^{1/p} - n}  \right] = \gamma. 
\end{equation*}
(ii) (parameter limit) Let $n \ge 2$. 
Let $\alpha \in \overline{\mathbb H}$. 
Then, \eqref{eq:positive-parameterlimit-gen} holds. 
\end{Cor}

\subsection{Proofs}

\begin{proof}[Proof of Theorem \ref{thm:negative-unbiased}]
We first show (i). 
We first deal with the case that $\alpha \in \mathbb H$. 
Since $p > -1$, we see that $\textup{Im}\left( (z + \alpha)^{p} \right) < 0$ for every $z \in \overline{\mathbb{H}}$ and hence, 
for every fixed $z_2, \cdots, z_n \in \overline{\mathbb{H}}$, 
$\sup_{z_1 \in \overline{\mathbb{H}}} \left| \frac{1}{n} \sum_{j=1}^{n} (z_j + \alpha)^p \right|^{\frac{1}{p}} < +\infty$,
and as a function of $z_1$, $\left(\frac{1}{n} \sum_{j=1}^{n} (z_j + \alpha)^p \right)^{\frac{1}{p}}$ is holomorphic on an open neighborhood on $\overline{\mathbb{H}}$.

By Theorem \ref{thm:negative-var}, 
we can apply Fubini's theorem and the Cauchy integral formula repeatedly. 
Thus we obtain that 
\begin{align}\label{eq:negative-Cauchyintegral-1}
E\left[ \left(\frac{1}{n} \sum_{j=1}^{n} Y_j^p \right)^{\frac{1}{p}}\right] &= \int_{\mathbb R^n} \left( \frac{1}{n} \sum_{j=1}^{n} (z_j + \alpha)^p \right)^{\frac{1}{p}} \prod_{j=1}^{n} \frac{\textup{Im}(\gamma)}{|z_j - \gamma|^2}  \prod_{j=1}^{n} dz_j  \notag\\
&=  \gamma + \alpha. 
\end{align}

We now deal with the case that $\alpha \in \mathbb R$. 
We can assume that $\alpha = 0$. 
for every fixed $z_2, \cdots, z_n \in \overline{\mathbb{H}} \setminus \{0\}$, 
\[ \lim_{|z_1| \to \infty; z_1 \in \overline{\mathbb{H}}} \left| \frac{1}{n} \sum_{j=1}^{n} z_j^p \right|^{\frac{1}{p}} = \left| \frac{1}{n} \sum_{j=2}^{n} z_j^p \right|^{\frac{1}{p}}, \textup{ and }  \lim_{|z_1| \to 0;  z_1 \in \overline{\mathbb{H}}} \left| \frac{1}{n} \sum_{j=1}^{n} z_j^p \right|^{\frac{1}{p}} = 0.\]
Hence, for every fixed $z_2, \cdots, z_n \in \overline{\mathbb{H}} \setminus \{0\}$, $\sup_{z_1 \in \overline{\mathbb{H}}} \left| \frac{1}{n} \sum_{j=1}^{n} z_j^p \right|^{\frac{1}{p}} < +\infty.$
By Theorem \ref{thm:negative-var}, 
we can apply Fubini's theorem and the Cauchy integral formula repeatedly. 
Thus we obtain \eqref{eq:negative-Cauchyintegral-1} and (i). 

We show (ii). 
We remark that if $X$ follows the Cauchy distribution, then, $1/X$ also follows the Cauchy distribution. 
Furthermore, all i.i.d.\ sums of Cauchy random variables are also Cauchy random variables. 
Hence, if $\alpha \in \mathbb R$, 
then, $n/\sum_{j=1}^{n} Y_j^{-1}$ also follows the Cauchy distribution. 
Thus we see that (ii) holds. 

We show (iii). 
Let $n \ge 3$. 
Then, by the inequality of the geometric mean and the harmonic mean, 
\[ E\left[ \left| \frac{n}{\sum_{j=1}^{n} Y_j^{-1}} \right| \right] \le \frac{1}{\textup{Im}(\alpha)} E\left[\frac{n}{\sum_{j=1}^{n} |Y_j|^{-2}} \right] \le  \frac{1}{\textup{Im}(\alpha)} E\left[ |Y_1|^{2/n} \right]^n < +\infty. \]

Let $n = 2$. 
We remark that the ratio of density functions between Cauchy distributions is bounded. 
Specifically, for every $\gamma_1, \gamma_2 \in \mathbb H$, there exists a constant $C$ such that 
$|p_{\gamma_2}(x)| \le C |p_{\gamma_1}(x)|$ for every $x \in \mathbb R$. 
Hence, by scaling, if $\alpha \in \mathbb H$, then, we can assume that $\alpha = \gamma = i$.

Since 
\[ \left|\frac{2}{(X_1 + i)^{-1} + (X_2 + i)^{-1}} \right| \le 2 \left|\frac{X_1 X_2}{X_1 + X_2 + 2i} \right| + 2, \]
it suffices to show that 
\[ E\left[ \left|\frac{X_1 X_2}{X_1 + X_2 + 2 i} \right|\right] < +\infty. \]
The ratios $1/X_1$ and $1/X_2$ follow the Cauchy distribution with parameter $i$ again. 
Let 
\[ F(x_1, x_2) :=  \frac{1}{\sqrt{(x_1 + x_2)^2 + 4 x_1^2 x_2^2}} \frac{1}{(x_1^2 +1)  (x_2^2 + 1)}. \]
Then, 
\[ E\left[ \left|\frac{X_1 X_2}{X_1 + X_2 + 2 i} \right|\right] = \frac{1}{\pi^2} \int_{\mathbb R^2} F(x_1, x_2) dx_1 dx_2. \]

It is easy to see that 
\[ \int_{\{|x_2| > 1\} \cup \{|x_1| > 1\}} F(x_1, x_2) dx_1 dx_2 < +\infty.\]
It suffices to show that 
\[  \int_{|x_1| \le 1, |x_2| \le 1} F(x_1, x_2) dx_1 dx_2 < +\infty.  \]
By symmetry and the change of variables $s = x_1/x_2$, 
we obtain that 
\[  \int_{|x_1| \le 1, |x_2| \le 1}F(x_1, x_2)  dx_1 dx_2  
= 2\int_{-1 < x_1 < 1, s \in \mathbb R} G(x_1, s) dx_1 ds. \]
where we let 
\[ G(x_1, s) := \frac{|s|}{\sqrt{(s+1)^2 + 4 x_1^2}} \frac{1}{(x_1^2 +1)  (x_1^2 + s^2)}.\]

Thus it suffices to show that 
\begin{equation}\label{eq:finite}
\int_{-1 < x_1 < 1, s \in \mathbb R} G(x_1, s) dx_1 ds < +\infty. 
\end{equation}

We divide this integral according to the values of the variable $s$. 
It holds that 
\[ \int_{-1 < x_1 < 1, |s - 1/4| > 1} G(x_1, s) dx_1 ds < +\infty.  \]

Since $\int_0^1 \frac{|s|}{x^2 + s^2} dx \le \frac{\pi}{2}$ if $s \ne 0$,
we see that
\[ \int_{-1 < x_1 < 1, |s| < 3/4} G(x_1, s) dx_1 ds \le 8 \int_{0 < x_1 < 1, 0 < |s| < 3/4} \frac{|s|}{x_1^2 + s^2} dx_1 ds \le 8\pi. \]

We finally see that 
\[ \int_{-1 < x_1 < 1, |s+1| \le 1/4} G(x_1, s) dx_1 ds = 4 \int_{|x_1| < 1, |u| \le 1/4} \frac{dx_1 du}{\sqrt{u^2 + 4 x_1^2}} < +\infty. \]

Thus we obtain \eqref{eq:finite} and assertion (iii). 
\end{proof}

As mentioned above, if $\alpha \in \mathbb H$, then, we can assume that $\alpha = \gamma = i$. 

\begin{proof}[Proof of Theorem \ref{thm:negative-Cauchy-var}]
(i)  
Let $n = 2$. 
Then, it holds that 
\[ E\left[ \left| \frac{Y_1^p + Y_2^p}{2} \right|^{2/p}\right] = \int_{\mathbb R^2} \frac{C_{p,1} dxdy}{\left|(x+i)^{-p} + (y+i)^{-p}\right|^{-2/p}} \ge  \int_{\mathbb R^2} \frac{C_{p,2} dxdy}{|x+i|^{2} + |y+i|^{2}} = +\infty.  \]
Thus we obtain (i). 

(ii) 
Assertion (ii) is easy to see for $\alpha \in \mathbb{R}$. 
We assume that $\alpha \in \mathbb H$. 
Then, 
\begin{align*}
E\left[ \left| \sum_{j=1}^{3} Y_j^{-1} \right|^2 \right] &= \int_{\mathbb R^3} \frac{dxdydz}{(xy+yz+zx-3)^2 + 4(x+y+z)^2} \\
&= \int_{\mathbb R} \frac{\pi^2}{\sqrt{3(z^2 + 4)}} dz = +\infty. 
\end{align*}
Thus we obtain (ii). 

(iii) This assertion follows from Theorem \ref{thm:negative-var} (i). 

(iv) 
Let $p = -1$. 
If $n \ge 5$, then the assertion follows from Theorem \ref{thm:negative-var} (ii). 
Assume that $n=4$.  
Then, by the change of variable $x_i = \tan \theta_j$, $1 \le j \le 4$, 
\[ E\left[ \left| \frac{1}{4} \sum_{j=1}^{4} Y_j^{-1} \right|^{-2} \right] 
\le E\left[ \left( \frac{4}{\sum_{j=1}^{4} |Y_j|^{-2}} \right)^2  \right]  = \frac{16}{\pi^4} \int_{\mathbb R^4} \frac{dx_1 \cdots dx_4}{\left( \sum_{j=1}^4 (x_j^2 + 1)^{-1}\right)^2 \prod_{j=1}^4 (x_j^2 + 1)} \]
\[ = \int_{(-\pi/2, \pi/2)^4} \frac{d\theta_1 \cdots d\theta_4}{(\sum_{j=1}^4 \cos^2 \theta_j)^2}. \]
Hence, for every $M > 0$, 
\[ E\left[ \left| \frac{1}{4} \sum_{j=1}^{4} Y_j^{-1} \right|^{-2}, \ \ \bigcup_{j=1}^4 \{|X_j| \le M\} \right] < +\infty. \]

Now it suffices to show that for some $M > 1$, 
\begin{equation}\label{eq:var4-WTS}
E\left[ \left| \frac{1}{4} \sum_{j=1}^{4} Y_j^{-1} \right|^{-2}, \ \ \bigcap_{j=1}^4 \{|X_j| > M\} \right] < +\infty. 
\end{equation}

Let 
\begin{align*}
F(x_1, x_2, x_3, x_4) &:= (x_1 x_2 x_3 - x_1 - x_2 - x_3)^2 + (x_1 x_2 x_4 - x_1 - x_2 - x_4)^2 \\
&+ (x_1 x_3 x_4 - x_1 - x_3 - x_4)^2 +(x_2 x_3 x_4 - x_2 - x_3 - x_4)^2. 
\end{align*}
Then, 
\[ E\left[ \left| \sum_{j=1}^{4} Y_j^{-1} \right|^{-2}, \ \ \bigcap_{j=1}^4 \{|X_j| > M\} \right]  \le  \int_{\bigcap_{j=1}^4 \{|x_j| > M\}} \frac{dx_1 \cdots dx_4}{F(x_1, x_2, x_3, x_4)}. \]

Let 
\begin{align*} 
G(y_1, y_2, y_3, y_4) &:= y_1^2 (1 - y_2 y_3 - y_3 y_4 - y_4 y_2)^2 + y_2^2 (1 - y_1 y_3 - y_3 y_4 - y_4 y_1)^2 \\
&+ y_3^2 (1 - y_1 y_2 - y_2 y_4 - y_4 y_1)^2 + y_4^2 (1 - y_1 y_2 - y_2 y_3 - y_3 y_1)^2. 
\end{align*}

By the change of variable $y_j = 1/x_j$, $1 \le j \le 4$, 
\[ \int_{\bigcap_{j=1}^4 \{|x_j| > M\}} \frac{dx_1 \cdots dx_4}{F(x_1, x_2, x_3, x_4)} = \int_{\bigcap_{j=1}^4 \{|y_j| < 1/M\}} \frac{dy_1 \cdots dy_4}{G(y_1, y_2, y_3, y_4)}. \]

We see that 
\[  G(y_1, y_2, y_3, y_4) \ge \frac{y_1^2 + y_2^2 + y_3^2 + y_4^2}{2},  \ \ \cap_{j=1}^4 \{|y_j| < 1/M\}, \]
for a large constant $M$. 
Then, 
\[  \int_{\bigcap_{j=1}^4 \{|y_j| < 1/M\}} \frac{dy_1 \cdots dy_4}{G(y_1, y_2, y_3, y_4)} \le 2\int_{\bigcap_{j=1}^4 \{|y_j| < 1/M\}} \frac{dy_1 \cdots dy_4}{y_1^2 + y_2^2 + y_3^2 + y_4^2} < +\infty. \]
Thus we obtain \eqref{eq:var4-WTS}. 
\end{proof}

\begin{proof}[Proof of Proposition \ref{prop:asymp-var-vary}]
(i) This follows from $X_1 \notin L^1$ and the definition of $V(p)$. 

(ii) This follows from l'Hospital's theorem and  the proof of \cite[Theorem 4.2]{Akaoka2021-1}. 

(iii) Recall $\alpha \in \mathbb{R}$.  
By \cite[Proposition 2.4 (2)]{Akaoka2021-3}, 
\[ V(p) = \frac{|\gamma + \alpha|^2}{p^2} \left(\frac{\cos(p \pi b)}{\cos(p \pi)} - 1 \right), \]
where we let $b := 2\textup{arg}(\gamma+\alpha)/\pi - 1 \in (-1, 1)$. 

Let $\{E_n (y) \}_{n \ge 0}$ be the {\it Euler polynomials}, that is, 
\[ \frac{2 e^{yt}}{e^t + 1} = \sum_{n=0}^{\infty} E_n (y) \frac{t^n}{n!}, \ |t| < \pi. \]

By substituting $2\pi i p$ and $(1+b)/2$ for $t$ and $y$ respectively in the above equation and taking the real parts of it, 
we see that 
\[ \frac{1}{p^2} \left(\frac{\cos(p \pi b)}{\cos(p \pi)} - 1 \right) = \sum_{n=1}^{\infty} (-1)^n E_{2n} \left(\frac{1+b}{2}\right) \frac{(2\pi p)^{2(n-1)}}{(2n)!}, \ \ |p| < \frac{1}{2}. \]

By \cite[(23.1.13) and (23.1.8)]{Abramowitz1965}, 
$(-1)^n E_{2n} \left(\frac{1+b}{2}\right) > 0$. 
Hence, $V$ is strictly convex and decreasing on  $(-1/2, 0)$. 

(iv) Recall that $\alpha = \gamma= i$. 
Then, 
$$V(p) = \frac{4^{1-p}}{p^2}  \left(\frac{1}{\sqrt{\pi}}  \frac{\Gamma(1/2 - p)}{\Gamma(1-p)} -4^p \right), \ \ -1 \le p < 0, $$
where $\Gamma$ is the gamma function, 
and hence, 
$V(-1) = 4$, which attains the Cram\'er-Rao lower bound. 
See also \cite[Remark 4.5 (ii)]{Akaoka2021-1}. 
Hence, $V(p) \ge V(-1), -1 \le p < 0$.  
\end{proof}

\begin{proof}[Proof of Theorem \ref{thm:inaccu}]
First, we show (i). 
We establish the multidimensional version of \cite[Lemma 1.14]{DenHollander2000} which states the smoothness of the rate function in one dimension. 
By the assumption, $f(X_1) - f(\gamma)$ is proper. 
Hence, by \eqref{eq:cov-nonsingular}, $\textup{Cov}(f(X_1)) =  \frac{\textup{Var}(f(X_1))}{2} I_2$ is nonsingular. 
Let 
$$ \Lambda(\lambda) := \log\left( E\left[ \exp(\braket{\lambda, f(X_1) - f(\gamma)}) \right] \right), \  \lambda \in \mathbb{R}^2, $$
where $\braket{,}$ is the standard inner product on $\mathbb R^2$ and we regard $f$ as an $\mathbb{R}^2$-valued function. 
Then, by the assumption, 
$\Lambda$ is finite for every $\lambda \in \mathbb{R}^2$.  
Hence, $\Lambda$ is smooth on $\mathbb R^2$. 
Furthermore, 
$\nabla \Lambda : \mathbb{R}^2 \to \mathbb{R}^2$ is a smooth map and the Hessian of $\Lambda$ at $0$, which is the Jacobian of $\nabla \Lambda$, $J(\nabla \Lambda)$, at $0$, is $\textup{Cov}(f(X_1))$.  
Hence by the inverse function theorem, 
there exist an open neighborhood $U$ of $0$ and an $\mathbb R^2$-valued smooth injective map $\lambda = \lambda(x)$ on $U$ such that 
$x = \nabla \Lambda(\lambda(x)), \ x \in U.$ 
Hence, $I_2 = J(\nabla \Lambda)(x) J(\lambda)(x), x \in U$, and furthermore, 
$$J(\lambda)(0) = (\textup{Cov}(f(X_1)))^{-1} = \frac{2}{\textup{Var}(f(X_1))} I_2.$$

Let $\Lambda^*$ be the Fenchel-Legendre transform of $\Lambda$. 
Then, 
\[ \Lambda^* (x) = \sup_{\lambda \in \mathbb{R}^2} \braket{\lambda,x} - \Lambda(\lambda) = \braket{\lambda(x),x} - \Lambda(\lambda(x)), \ \ x \in U. \]
We see that $\Lambda^* (0) = 0$ and $\nabla \Lambda^* = \lambda$ on $U$. 
Since $\lambda$ is injective and $\nabla \Lambda (0) = 0$, 
$\nabla \Lambda^* (0) = \lambda(0) = 0$. 
The Hessian of $\Lambda^*$ at $0$ is $\frac{2}{\textup{Var}(f(X_1))} I_2$. 
Therefore, 
\[ \lim_{x \to0} \frac{\Lambda^* (x)}{|x|^2} = \frac{1}{\textup{Var}(f(X_1))}. \]
By this, $\Lambda^* (0) = 0$, and the convexity and non-negativity of $\Lambda^*$, 
\[ \lim_{\varepsilon \to +0} \frac{1}{\varepsilon^2} \inf_{x \notin B(0,\varepsilon)}  \Lambda^* (x)  = \lim_{\varepsilon \to +0} \frac{1}{\varepsilon^2} \inf_{|x| =\varepsilon}  \Lambda^* (x)  = \frac{1}{\textup{Var}(f(X_1))}.  \]
By Cram\'er's theorem, it holds that for every sufficiently small $\varepsilon > 0$, 
\[ \lim_{n \to \infty} \frac{\log P\left( \left|\frac{1}{n} \sum_{i=1}^{n} f(X_i)  - f(\gamma) \right| > \varepsilon \right)}{n} = -\inf_{x \in B(0,\varepsilon)^c} \Lambda^* (x). \]
Therefore, 
\begin{equation}\label{eq:intermediate-LDP}
 \lim_{\varepsilon \to +0} \frac{1}{\varepsilon^2} \left( \lim_{n \to \infty} \frac{\log P\left( \left|\frac{1}{n} \sum_{i=1}^{n} f(X_i)  - f(\gamma) \right| > \varepsilon \right)}{n} \right) = -\frac{1}{\textup{Var}(f(X_1))}. 
\end{equation}

Since $f$ is holomorphic, 
it holds that for every $\eta \in (0,1)$, 
there exists $\varepsilon_0 > 0$ such that for every $\varepsilon \in (0, \varepsilon_0)$, 
\[ B\left(f(\gamma), (1-\eta)|f^{\prime}(\gamma)|\varepsilon\right) \subset f(B(\gamma, \varepsilon)) \subset B\left(f(\gamma), (1+\eta)|f^{\prime}(\gamma)|\varepsilon\right). \]
By this and \eqref{eq:intermediate-LDP}, we have assertion (i). 

By the moderate deviation principle \cite[Theorem 3.7.1]{Dembo2010}, we can show (ii) in the same manner as in the proof of (i). 
\end{proof}

\begin{proof}[Proof of Corollary \ref{cor:truncated-Cauchy}]
(i) By Lemma \ref{lem:positive-fundamental}, 
we see that for every fixed $\alpha \in \overline{\mathbb{H}}$ and $z_2, \cdots, z_n \in \mathbb R$, 
\[ \left| \left(\sum_{j=1}^{n} (z_j + \alpha)^{p}\right)^{\frac{1}{p}} - \sum_{j=1}^{n} (z_j + \alpha) \right| = O\left(|z_1|^{\max\{p,1-p\}}\right), \ \ \ z_1 \in \overline{\mathbb H}, \  |z_1| \to \infty. \]
Now the assertion follows from Theorem \ref{thm:truncated-integrability}, 
Fubini's theorem, and the repeated uses of the Cauchy integral formula. 
See also the remark below. 

(ii) It is easy to see that \eqref{ass-fine} holds for the Cauchy distribution, and hence the assertion follows from Theorem \ref{thm:positive-parameterlimit}. 
\end{proof}

\begin{Rem}
In the above proof, 
in order to apply the Cauchy integral formula for the case that $\alpha \in \mathbb{R}$, 
it is convenient to take an {\it unusual branch cut} for $z \in -\mathbb H$. 
Specifically, we need to let 
$$\log(z) := \log r + i\theta, \ \ z = r\exp(i\theta), \theta \in ((\varepsilon-1)\pi, (1+\varepsilon)\pi], $$
for some sufficiently small $\varepsilon > 0$. 
\end{Rem}

\section{Point estimation for parameters of the mixture Cauchy model}\label{sec:mix}

If we use complex-valued positive power means, then it is easy to construct a strongly consistent estimator of the parameters of mixture Cauchy models. 
The fractional moment is useful for estimations of parameters of some one-dimensional distributions (\cite{Tallis1968, From1989}).  

\begin{Def}[{\cite[pp.480-481]{Lehmann1999}}]
If the probability density function is given by 
\[ \frac{1-t}{\pi} \frac{\sigma_1}{(x-\mu_1)^2 + \sigma_1^2} + \frac{t}{\pi} \frac{\sigma_2}{(x-\mu_2)^2 + \sigma_2^2}, \]
for some $0 < t < 1$ and $(\mu_1, \sigma_1) \ne (\mu_2, \sigma_2)$, 
then, we call the model the {\it mixture Cauchy model} $C(t; \mu_1, \sigma_1, \mu_2, \sigma_2)$. 
We remark that this model is symmetric, in the sense that we can replace $t$ with $1-t$, $(\mu_1, \sigma_1)$ with $(\mu_2, \sigma_2)$, and $(\mu_2, \sigma_2)$ with $(\mu_1, \sigma_1)$. 
\end{Def}

Now we give two strongly consistent and $\sqrt{n}$-consistent estimators of each of the five parameters $(t, \mu_1, \sigma_1, \mu_2, \sigma_2)$ in closed forms.
\cite[pp.480-481]{Lehmann1999} deals with a point  estimation of the weight $t$ when $\mu_1, \sigma_1, \mu_2, \sigma_2$ are all known.  
Generally,  straightforward applications of maximal likelihood estimation or order statistics do not work well in mixture parametric models. 
Instead, the Expectation-Maximization (EM) algorithm is often adopted. 
Kalantan and Einbeck  \cite{Kalantan2019}  used a version of the EM algorithm with appropriately weighted quantiles. 
However, \cite{Kalantan2019} focuses on simulation study and does not give any mathematical proof of the convergence. 
We can characterize the mixture Cauchy model by power means. 
If \eqref{mix-mean} below holds for every $\beta$ in a set of positive numbers containing a convergent sequence, 
then, $X$ follows $C(t; \mu_1, \sigma_1, \mu_2, \sigma_2)$. 
See \cite[Theorem 1]{Lin1992}\footnote{We can easily extend the result for the case that $X$ is not non-negative, since $X^{\alpha} = (X^+)^{\alpha} + \exp(i \pi \alpha)(X^{-})^{\alpha}$ and $ \sin(\pi \alpha) \ne 0$ if $0 < |\alpha| < 1$.} and \cite[Corollary 3.7]{Okamura2020}.

By the Cauchy integral formula, we see that 
\begin{equation}\label{mix-mean} 
(1-t) (\mu_1 + \sigma_1 i)^{\beta} + t (\mu_2 + \sigma_2 i)^{\beta} = E\left[X_1^{\beta} \right], \ \ \beta \in (0,1). 
\end{equation} 

Let $0 < \alpha < 1/6$. 
Let $a_1 = (\mu_1 + \sigma_1 i)^{\alpha}$ and  $a_2 = (\mu_2 + \sigma_2 i)^{\alpha}$. 
Let $B_j := E\left[X_1^{j \alpha}\right], 1 \le j \le 3$. 
Then, by \eqref{mix-mean}, 
\[ (1-t) a_1^j + t a_2^j = B_j,  \ 1 \le j \le 3. \]

We see that $a_1 \ne a_2$ and $t = \dfrac{a_1 - B_1}{a_1 - a_2}$. 
Since $\mu_j + \sigma_j i = a_j^{1/\alpha}$, $j=1,2$, 
it suffices to obtain strongly consistent and $\sqrt{n}$-consistent estimators for $a_1$ and $a_2$. 
We see that 
\[ B_2 - B_1^2 = t(1-t) (a_1-a_2)^2, \]
\[ B_3 - B_1 B_2 = t (1-t) (a_1 + a_2) (a_1 - a_2)^2, \]
and, 
\[ B_1 B_3 - B_2^2 = t(1-t) a_1 a_2 (a_1 - a_2)^2.\]

Since $t(1-t)(a_1 - a_2)^2 \ne 0$, 
\[ a_1 + a_2 = \frac{B_3 - B_1 B_2}{B_2 - B_1^2},  \]
and, 
\[ a_1 a_2 = \frac{B_1 B_3 - B_2^2}{B_2 - B_1^2}. \]
Hence, an expression of $(a_1, a_2)$ is given by 
\[ \begin{cases} a_1 = F_3 (B) = F_1 (B) + \sqrt{F_2 (B)}, \\ 
a_2 = F_4 (B) = F_1 (B) - \sqrt{F_2 (B)}. \end{cases} \]
where we let 
$B := (B_1, B_2, B_3)$, 
\[ F_1 (x_1, x_2, x_3) := \frac{x_1 x_3 - x_2^2}{2(x_2 - x_1^2)},\]
\[ F_2 (x_1, x_2, x_3) := F_1 (x_1, x_2, x_3)^2 - \frac{x_1 x_3 - x_2^2}{x_2 - x_1^2}, \]
\[ F_3  (x_1, x_2, x_3) := F_1  (x_1, x_2, x_3)+ \sqrt{F_2 (x_1, x_2, x_3)}, \]
and 
\[  F_4  (x_1, x_2, x_3) := F_1 (x_1, x_2, x_3) - \sqrt{F_2 (x_1, x_2, x_3)}.  \]

We define a map $\psi : \mathbb{R} \to \mathbb{C}^3$ by 
$\psi(x) := \left(x^{\alpha}, x^{2\alpha}, x^{3\alpha}\right)$. 
Then, by the strong law of large numbers, 
$\displaystyle \left(\overline{X}_n := \frac{1}{n} \sum_{j=1}^{n} \psi(X_j)\right)_n$ converges to $B$, as $n \to \infty$, almost surely, 
and furthermore, by the multidimensional central limit theorem,  
$\left(\sqrt{n}(\overline{X}_n - B)\right)_n$ converges weakly to a $6$-dimensional normal distribution $N(0, \Sigma)$ as $n \to \infty$, 
where $\Sigma$ is the variance-covariance matrix of $\psi(X_1)$, 
which is a non-negative definite matrix of degree $6$.   
Here the assumption that $0 < \alpha < 1/6$ is used.

 Now we obtain a strongly consistent estimator $\hat{a_{i, n}}$ of $a_i$, $i=1,2$,   
if we replace all $B_j$'s with their consistent estimators $\displaystyle\left( \frac{1}{n} \sum_{\ell=1}^{n} X_{\ell}^{j\alpha}  \right)_n$, $1 \le j \le 3$.  
 Specifically, we let 
 \[ \hat{a_{1}}_{n} := F_3 (\overline{X}_n) = F_1 (\overline{X}_n) + \sqrt{F_2 (\overline{X}_n)} \]
 and 
 \[ \hat{a_{2}}_{n} := F_4 (\overline{X}_n) = F_1 (\overline{X}_n) - \sqrt{F_2 (\overline{X}_n)}, \]
 where we take the principle value for the square root. 
 Since the set of solutions of an algebraic equation is continuous with respect to its coefficients, 
 $\left(\left\{\hat{a_{1}}_{n}, \hat{a_{2}}_{n}\right\}\right)_n$ converges to the set $\{a_1, a_2\}$, as $n \to \infty$ with respect to the Hausdorff distance on $\mathbb{R}^2$,  almost surely. 
 
A sequence of sets $(\left\{a_{1,n}, a_{2,n}\right\})_n$ converges to the set $\{a_1, a_2\}$ as $n \to \infty$, with respect to the Hausdorff distance, 
if and only if 
$H (\{a_{1,n}, a_{2,n}\}, \{a_1, a_2\}) \to 0, \ n \to \infty$,  
where we let  $$ H(\{a,b\}, \{c,d\}) :=\max\left\{\min\{|a-c|, |a-d|\}, \min\{|b-c|, |b-d|\}\right\},  \ a,b,c,d \in \mathbb{C}.$$
We have $H (\{\hat{a_{1}}_{n}, \hat{a_{2}}_{n}\}, \{a_1, a_2\}) \to 0, \ n \to \infty$, almost surely.

We give some numerical computations by using the software R.  
Let $\alpha = 1/10$. 
Consider the cases that $(\mu_1, \sigma_1, \mu_2, \sigma_2) = (0,1,20,2)$ and that $(\mu_1, \sigma_1, \mu_2, \sigma_2) = (0,1,5,6)$. 
For the sizes of samples and the weights, we consider the following 12 cases that $n=100, 1000, 10000$ and $t = 1/6, 1/4, 1/3, 1/2$. 
We compute $H(\{\hat{a_{1}}_{n}, \hat{a_{2}}_{n}\},\{a_1, a_2\})$ for  $10^4$ samples and consider the mean, which approximates the expectation $E\left[H(\{\hat{a_{1}}_{n}, \hat{a_{2}}_{n}\},\{a_1, a_2\})\right]$. 

\begin{table}[H]
\[
\begin{array}{c|c|c|c|c}
		     & 1/6 & 1/4 & 1/3  &1/2  \\
		\hline
		100      & 0.162 & 0.114 &  0.092 & 0.080 \\
		1000      & 0.073 & 0.047 &  0.036 & 0.030 \\
		10000    &  0.025 & 0.017 &  0.013 & 0.010
		\end{array}
\]
\caption{$(\mu_1, \sigma_1, \mu_2, \sigma_2) = (0,1,20,2)$}
		\label{table:1}
\end{table}

\begin{table}[H]
\[
\begin{array}{c|c|c|c|c}
		     & 1/6 & 1/4 & 1/3  &1/2  \\
		\hline
		100      & 0.549 & 0.470 &  0.433 & 0.449 \\
		1000      & 0.234 & 0.181 &  0.157 & 0.152 \\
		10000    &  0.088 & 0.065 & 0.055 & 0.045 
		\end{array}
\]
\caption{$(\mu_1, \sigma_1, \mu_2, \sigma_2) = (0,1,5,6)$}
		\label{table:2}
\end{table}

We finally consider  $\sqrt{n}$-consistent estimators of $(a_1, a_2)$. 
The map $z \mapsto \sqrt{z}$ is a measurable on $\mathbb{C}$ and holomorphic on $\mathbb{C} \setminus (-\infty, 0]$. 
We see that $F_2 (B) = (a_2 - a_1)^2$. 
If $\textup{Re}(a_1 - a_2) \ne 0$, then, $(a_2 - a_1)^2 \notin (-\infty, 0]$. 
Now by the delta method, $(\sqrt{n}(\hat{a_{1}}_{n} - a_1))_n$ converges weakly to a $2$-dimensional normal distribution $N(0, \Sigma^{(1)})$ as $n \to \infty$, 
where $\Sigma^{(1)}$ is a non-negative definite matrix of degree $2$. 
In the same manner, we see that $\hat{a_{2}}_{n}$ is a $\sqrt{n}$-consistent estimator of $a_2$. 

We consider the case that $\textup{Re}(a_1 - a_2) = 0$. 
Let $$F_5 (x_1, x_2, x_3) := \dfrac{F_1 (x_1, x_2, x_3)^2}{F_2 (x_1, x_2, x_3)} -2 \ \textup{ and } F_6 (x_1, x_2, x_3) := F_5 (x_1, x_2, x_3)^2 - 4.$$
Then, $F_5 (B) = \dfrac{a_1}{a_2} + \dfrac{a_2}{a_1}$, and $\displaystyle F_6 (B) = \left(\frac{a_1}{a_2} - \frac{a_2}{a_1} \right)^2$. 

If $\textup{Re}(a_1 - a_2) = 0$, then, $|a_1| \ne |a_2|$ and hence $\left(\dfrac{a_1}{a_2} - \dfrac{a_2}{a_1} \right)^2 \notin (-\infty, 0]$.  
Thus, in the same manner, as above,  we have a $\sqrt{n}$-consistent estimator of $a_2 / a_1$, and hence, we also have a $\sqrt{n}$-consistent estimator of $a_1 = (a_1 + a_2)/(1+ a_2/a_1)$. 
The condition that $|a_1| \ne |a_2|$ is equivalent with $\mu_1^2 + \sigma_1^2 \ne \mu_2^2 + \sigma_2^2$, which does not depend on $\alpha$. \\

{\it Acknowledgements} \ The authors wish to express our gratitude to an anonymous referee for his or her comments to improve the paper. 
The second author was supported by JSPS KAKENHI 19K14549 and 22K13928, and, the third author was supported by JSPS KAKENHI 16K05196 and 23K03213. 


\end{document}